\documentclass[a4paper,11pt]{article}
\usepackage{amsmath}
\usepackage{amssymb}
\usepackage{color}
\usepackage{graphicx}
\usepackage{epsfig}
\usepackage{epstopdf}
\usepackage[applemac]{inputenc}

\newenvironment{proof}[1][Proof]{\textbf{#1.} }{\ \rule{0.5em}{0.5em}}
\newtheorem{remark}{\textbf{Remark}}[section]

\newtheorem{proposition}{\textbf{Proposition}}[section]
\newtheorem{lemma}{\textbf{Lemma}}[section]
\newtheorem{theorem}{\textbf{Theorem}}[section]

\numberwithin{equation}{section}

\def\dd{{\rm d}}

\newcommand{\ov}[1]{\overline{#1}}

\def\weight(#1,#2){c_{#1,#2}}
 \if {
  
 } \fi
\if {

}{{\caption}}
} \fi

\def\db{\bar{d}}

\def\B{\mathcal{B}}
\def\C{\mathcal{C}}

\def\F{\mathcal{F}}
\def\G{\mathcal{G}}

\def\K{\mathcal{K}}

\def\P{\mathcal{P}}

\def\SS{\mathcal{S}}

\def\eps{\varepsilon}

\def\half{\mbox{$\frac{1}{2}$}}

\def\1B{{\bf  1}}

\newcommand{\NN}{\mathbb{N}}
\newcommand{\ZZ}{\mathbb{Z}}

\newcommand{\RR}{\mathbb{R}}

\def\EE{\mathbb{E}}
\def\PP{\mathbb{P}}

\newcommand\be{\begin{equation}}
\newcommand\ee{\end{equation}}
\newcommand\ba{\begin{array}}
\newcommand\ea{\end{array}}
\newcommand{\bean}{\begin{eqnarray*}}
\newcommand{\eean}{\end{eqnarray*}}

\def\ds{\displaystyle}

\title{A  Semi-Lagrangian scheme for a degenerate second order Mean Field Game system}

\author{ E. Carlini \thanks{Dipartimento di Matematica, Sapienza Universit\`a di Roma (carlini@mat.uniroma1.it)  +39 06 49913214.} \and F. J. Silva  \thanks{XLIM - DMI 
UMR CNRS 7252
Facult\'e des Sciences et Techniques,
Universit\'e de Limoges  (francisco.silva@unilim.fr) +33 5 87506787. The support 
 of the European Union under the ``7th Framework Program FP7-PEOPLE-2010-ITN Grant agreement number 264735-SADCO'' is gratefully acknowledged.}}
\begin{document}
\maketitle
\begin{abstract} In this paper we study a fully discrete Semi-Lagrangian approximation of a second order Mean Field Game system, which can be degenerate. We prove that the resulting scheme is well posed and, if the state dimension is equals to one, we prove a convergence result. Some numerical simulations are provided, evidencing  the convergence of the approximation and also the difference between the numerical results for the degenerate and non-degenerate cases. 
\end{abstract}


\thispagestyle{plain}
%%%%%%%%%%%%%%%%%%%%%%%%%%%%%%%%%%%%%%%%%%%%%%%%%%%%%%%%%%%%%%%%%%%%%%%%%%%%%%%%
\section{Introduction}
Mean Field Games (MFG) systems were introduced independently by \cite{HuangCainesMalhame03,HuangCainesMalhame06} and \cite{LasryLions06i,LasryLions06ii,LasryLions07} in order to model dynamic games with a large number of {\it indistinguishable small players}. In the model proposed in  \cite{LasryLions06ii,LasryLions07} the asymptotic equilibrium is described by means of a system of two Partial Differential Equations (PDEs). The first equation, together with a final condition, is a Hamilton-Jacobi-Bellman (HJB) equation describing the value function of an {\it average player} whose cost function depends on the distribution $m$ of the entire population. The second equation is a Fokker-Planck equation which, together with an initial distribution $m_0$,  describes the fact that $m$ evolves following  the optimal dynamics of the average player. We refer the reader to the original papers  \cite{HuangCainesMalhame03,HuangCainesMalhame06,LasryLions06i,LasryLions06ii,LasryLions07} and the surveys \cite{Cardialaguet10,gomessurvey13} for a detailed description of the problem and to \cite{GLL10} for some interesting applications. 

Numerical methods to solve MFGs problems have been addressed by several authors. Let us mention the papers  \cite{AchdouCapuzzo10, aime10, gueant10,AchdouCapuzzoCamilli12,CS13} where the second order system (i.e. when the underlying dynamics is stochastic) is treated and to \cite{camsilva12,CS12} for the first order case (i.e. when the underlying dynamics is deterministic).

In this article we consider the following second order {\it possibly degenerated} MFG system 
\small
\be\ba{rcl}\label{MFGesto} 
-\partial_{t} v  -  \half \mbox{tr}\left( \sigma(t) \sigma(t)^{\top} D^{2} v \right)  + \half |D v|^{2}   &=& F(x, m(t)) \;  \;    \hbox{in } \RR^{d}\times ]0,T[, \\[6pt]
\partial_{t} m -\half  \mbox{tr}\left( \sigma(t) \sigma(t)^{\top} D^{2} v \right)   -\mbox{div} \big( D v m  \big)  &=&0 \;  \; \; \hbox{in } \RR^{d}\times ]0,T[, \\[6pt]
v(x,T)= G(x, m(T)) \; \;   \mbox{for } x \in \RR^{d}, &\,& \; \; m(\cdot,0)=m_0(\cdot) \in \P_{1}(\RR^{d}),
\ea\ee \normalsize
where $\P_{1}(\RR^{d})$ is the set of probability measures over $\RR^{d}$ having finite first order moment,  $\sigma: [0,T] \to \RR^{d\times r}$ and $F$, $G : \RR^{d} \times \P_{1} \to \RR$ are two functions satisfying some assumptions described in Section \ref{preliminares}. Up to the best of our knowledge, for this system,  existence and uniqueness results have  not been established yet (except for the case $r=d$, $\sigma:= \hat{\sigma} \mathbb{I}_{d\times d}$, $\hat{\sigma}\in \RR$).

The  aim of this work is to provide a fully-discrete Semi-Lagrangian discretization of \eqref{MFGesto}, to study  the main properties of the scheme and to establish a convergence result for the solutions of the discrete system. The line of argument is similar to the one analyzed in  \cite{CS12}.  Given a continuous measure-valued   application  $\mu(\cdot)$ and a space-time step $(\rho,h)$ we discretize the HJB
\be\label{eqejemplo}
\ba{rcl}
-\partial_{t} v  -     \half \mbox{tr}\left( \sigma(t) \sigma(t)^{\top} D^{2} v \right)  + \half | D v|^{2}   &=& F(x, \mu(t)) \;  \;    \hbox{in } Q, \\[6pt]
v(x,T)&=& G(x, \mu(T)) \; \;   \mbox{for } x \in \RR^d,
\ea\ee
using a fully-discrete Semi-Lagrangian scheme in the spirit of \cite{CamFal95,DebrabantJakobsen13}. We then regularize the solution of the scheme by convolution with a mollifier $\phi_{\eps}$ ($\eps>0$). The resulting function is called $v^{\eps}_{\rho,h}[\mu]$. In order to discretize the second equation we propose a natural extension to the second order case  of the scheme in \cite{CS12} designed for the first order  equation (i.e. with $\sigma=0$). The solution of the scheme is denoted by $m^{\eps}_{\rho,h}[\mu](\cdot)$.  The fully-discretization of problem \eqref{MFGesto} is thus to find $\mu(\cdot)$ such that  $m^{\eps}_{\rho,h}[\mu](\cdot)=\mu(\cdot)$. The existence of a solution of the discrete problem is established in Theorem \ref{existenciamfgesto}  by standard arguments based on the Brouwer fixed point Theorem. The convergence of the solutions of the discrete system to a solution of  \eqref{MFGesto} is much more delicate. As a matter of fact, as in \cite{CS12} we   establish in Theorem \ref{resultadoprincipalfullydiscrete} the convergence result only when the state dimension $d$ is equals to one.  Under suitable conditions over the discretization parameters, the proof is based on three crucial results. The first one is a relative compactness property for  $m^{\eps}_{\rho,h}[\mu](\cdot)$, which can be obtained as a consequence of a Markov chain interpretation of the scheme.  The second result is  the discrete semiconcavity of $v^{\eps}_{\rho,h}[\mu]$ (see e.g. \cite{AchdouCamilliCorrias11}), which implies a.e. convergence of $Dv^{\eps}_{\rho,h}[\mu]$ to $Dv[\mu]$ (where $v[\mu]$ is the unique viscosity solution of \eqref{eqejemplo}). The third result are $L^{\infty}$-bounds for the density of  $m^{\eps}_{\rho,h}[\mu](\cdot)$, where the one dimensional assumption plays an important role. We remark that our convergence result proves the existence of a solution of \eqref{MFGesto}  when $d=1$. Moreover, our results are valid for more general Hamiltionians, as the ones considered in \cite{AchdouCamilliCorrias11} (see Remark \ref{extensionahamiltonianosmasgenerales}{\rm(ii)}).  However, since the proofs are already rather technical, as in \cite{CS12},  we preferred to present the details for the quadratic Hamiltonian case. 
%%$m^{\eps}_{\rho,h}[\mu](\cdot)$, under some conditions over the discretization parameters. This result is a consequence of  

The paper is organized as follows.  In Section \ref{preliminares} we fix some notations and we state our main assumptions. In Section \ref{mamdmoootroototootjkask} we provide the natural Semi-Lagrangian discretization for the HJB equation and we prove its main properties. In Section \ref{mnnwnrnwaa} we propose a scheme for the Fokker-Planck equation and we prove that the associated solutions, as functions of the discretization parameters,  form a relatively compact set. In Section \ref{qmeqmenresultados} we prove our main results, the existence of a solution of the discrete system and, if $d=r=1$, the convergence to a solution of \eqref{MFGesto}. Finally,  in Section \ref{numericaltests} we present some numerical simulations showing the difference between the numerical approximation between  degenerate and non-degenerate systems.
%where
%$$ a_{ij}(x,t):= (\sigma(x,t) \sigma(x,t)^{\top})_{ij}= {\color{red} \sigma(x,t)_{i \cdot} \sigma(x,t)_{j \cdot}} ,$$
%and we have used Einstein notation (repeated indexes are added).  
%{\color{magenta} Nevertheless,  our schemes, as well as their main properties,  can be extended to more general problems (see e.g. section \ref{weoqoeqooass}).} {\color{red} Specially more general Hamiltonian but assuming periodicity \cite{AchdouCapuzzoCamilli12}}
%
%{\color{magenta} The article is organized as follows: We begin in section \ref{wwrweqqn} by recalling the standard assumptions for the data of \eqref{MFGgeneral}. Next, we review in section \ref {weweowoeerew} the SL scheme introduced in   \cite{CS12} for the first order case  and the main results: existence of a solution of the scheme and convergence to the solution of \eqref{MFGgeneral}. In section \ref{fulldisc} we propose a SL scheme for the second order case and we provide in section  \ref{numericaltests} some numerical tests for both types of systems. Finally, we discuss in section \ref{weoqoeqooass} some natural generalizations of our scheme for more complicated problems.}
%\section{Preliminaries on the continuous problem}\label{wwrweqqn}
\section{Preliminaries}\label{preliminares}
Let us first fix some notations. 
 For $x\in \RR^d$ we will denote by $|x|= \sqrt{x^{\top}x}$ for the usual Euclidean norm.  In the entire article $c>0$ will be a generic constant, which can change from line to line. For  $u\in  \RR^{d}\times [0,T]\to \RR$  we will denote by $\partial_{t} u$  for the partial derivative of $u$ (if it exists) w.r.t. the time variable and   by $D u$, $D^{2}u$  the gradient and Hessian of $u$ (if they exist) w.r.t. the space variables. We denote by  $\P(\RR^{d})$   the set of Borel probability measures $\mu$  over $\RR^{d}$  and, for $p\in [1, \infty[$, we say that $\mu \in \P_{p}(\RR^{d})$  if
$$ \int_{\RR^{d}} |x|^{p} \dd \mu(x) < + \infty.$$
The distance $d_{p} : \P_{p}(\RR^{d})\times \P_{p}(\RR^{d})\to \RR$ is defined as \footnotesize
$$ d_{p}(\mu_1, \mu_2):= \inf_{\gamma \in \P(\RR^{d}\times \RR^{d})}\left\{ \left[\int_{\RR^{d}\times \RR^{d}}|x-y|^{p} \dd \gamma(x,y)\right]^{\frac{1}{p}} \; ; \;   \gamma(A\times \RR^{d})= \mu_1(A), \; \; \gamma( \RR^{d}\times B)= \mu_2(B)  \; \; \forall \; A, B \in \B(\RR^{d})\right\}.$$ \normalsize
It is well-known (see e.g.  \cite[Theorem 1.14]{Villani03}) that $d_1$,  can be expressed in the following dual form 
\be\label{formadualkantorovic}  d_{1}(\mu_{1}, \mu_{2}) = \sup_{\phi}\left\{ \int_{\RR^{d}} \phi(x) \dd [\mu_{1}-\mu_{2}](x) \ ; \ \phi \; \mbox{is 1-Lipschitz}\right\}.
\ee
Let us recall the following useful result (see e.g. \cite[Chapter 7]{Ambrosiogiglisav} and \cite[Lemma 5.7]{Cardialaguet10}):
\begin{lemma}\label{nwrnnnndbbruua}
Let $q> p >0$ and $\K\subseteq \P_{p}(\RR^{d})$ be such that 
$$ \sup_{\mu \in \K} \int_{\RR^{d}} |x|^{q} \dd \mu(x) < \infty.$$
Then  $\K$ is a relatively compact set in $\P_{p}(\RR^{d})$.
\end{lemma}
We assume now the  following assumptions on the data of \eqref{MFGesto}:
\medskip\\
{\bf(A1)} We suppose that: \smallskip\\
{\rm{\bf(i)}} $F$ and $G$ are  uniformly bounded   over $\RR^{d}\times \P_{1}$ and for every $m\in \P_{1}(\RR^{d})$, the functions $F(\cdot, m)$, $G(\cdot,m)$ are $C^{2}$ and their first and second derivatives are bounded in $\RR^{d}$, uniformly with respect to $m$, i.e. $\exists \;   c>0$ such that 
 $$  \|F(\cdot, m)\|_{C^{2}} +  \|G(\cdot, m)\|_{C^{2}} \leq  c \hspace{0.5cm} \forall \; m\in \P_{1}(\RR^{d}),$$
 where for $\phi: \RR^{d} \to \RR$ we set $\| \phi\|_{C^{2}}:= \sup_{x\in \RR^{d}, \; |\alpha|\leq 2} \left|D^{\alpha} \phi(x)\right|$.\smallskip\\
 {\rm {\bf(ii)}} Denoting by $\sigma_\ell: [0,T] \to \RR^{d}$ ($\ell=1,\hdots, r$) the $\ell$ column vector of the matrix $\sigma$, we assume that $\sigma_{\ell}$ is continuous. \smallskip\\
{\rm{\bf(iii)}}  The measure $m_{0}$ is absolutely continuous, with density still denoted as $m_0$. Moreover, we suppose that $m_{0}$ is essentially bounded and has compact support, i.e. there exists $c>0$ such that    $\mbox{supp}(m_{0}) \subseteq B(0, c)$, where $B(0, c):= \{ x\in \RR^{d} \; ; \; |x| < c\}$.
\smallskip

We say that $(v,m)$ is a solution of \eqref{MFGesto}  if the first equation  is satisfied  in the viscosity sense (see e.g. \cite{CraIshLio92,FleSon92}), while the second one  is satisfied in the distributional sense (see e.g \cite{Figalli08}), i.e.   for every $\phi \in \C_{c}^{\infty}\left(\RR^{d} \right)$ and $t\in [0,T]$
$$ \int_{\RR} \phi(x) \dd m(t)(x)=  \int_{\RR} \phi(x) \dd m_{0}(x)   +\int_{0}^{t}\int_{\RR^{d}} \left[ \half \mbox{Tr}(\sigma \sigma^{\top}(s) D^{2}\phi(x)) - \langle D \phi(x), Dv(x,s)\rangle\right] \dd m (s)(x)\dd s.$$
%{\bf(A2)}  We suppose that 
%\begin{itemize}
%\item[(i)] (Joint continuity) The applications $b:\RR^{d}\times [0,T] \times \RR^{m} \to \RR^{d}$ and $f:\RR^{d}\times [0,T] \times \RR^{m} \to \RR$ are continuous.
%\item[(ii)] (Dependence in the space variable) There exists a constant ${\color{red} c>0}$ such that $\forall \; (t,u)\in [0,T]\times U$
%\be\label{mwmemwemwmemwmemwesss}| f(x_1,t, u) -f(x_2,t, u) | +  |b(x_1,t, u) -b(x_2,t, u) |  \leq c |x_1- x_2| \hspace{0.5cm} \forall \; \; x_1, \; x_2 \in \RR^{d}.\ee
%\end{itemize}

Our aim in this work is to provide a discretization scheme for \eqref{MFGesto}.  Given $h,\rho>0$, let us define a space grid  $\mathcal{G}_{\rho}$ and  a time-space grid $\mathcal{G}_{\rho,h}$ as
$$\mathcal{G}_{\rho}  :=\{ x_i=i \rho,\; i\in \ZZ^{d} \}, \hspace{0.8cm} \mathcal{G}_{\rho,h}:= \mathcal{G}_{\rho}\times \{ t_k \}_{k=0}^{N},$$ where $t_k=k h$ ($k=0,\hdots, N$) and $t_{N}=Nh=T$. 
We call   $B(\mathcal{G}_{\rho})$ and   $B(\mathcal{G}_{\rho,h})$    the spaces of bounded functions defined respectively  on $\mathcal{G}_{\rho}$  and  $\mathcal{G}_{\rho,h}$. For  $f\in  B(\mathcal{G}_{\rho})$ and $g\in B(\mathcal{G}_{\rho,h})$ we set
$ f_{i}:= f(x_{i}),\, g_{i,k} := g(x_{i}, t_{k})$. Given a regular triangulation of $\RR^{d}$ with   vertices  belonging to $\mathcal{G}_{\rho}$, we set $\beta_{i}(x)$ for the barycentric coordinate of $x$ relative to $x_i$ in the triangulation.  Clearly $\beta_i(x)$  is a piecewise affine function with   compact support, satisfying  $0\leq \beta_i \leq 1$, $\beta_i(x_j)=\delta_{i j}$ for all $x_j \in \mathcal{G}_{\rho}$   (the Kronecker symbol) and $\sum_{i\in\ZZ^{d} }\beta_i(x)=1$ for all $x\in \RR^{d}$. We consider the following  linear interpolation operator  %%%%%%%%%%%%%%%%%%%%%%%%%%%%%%
\be\label{definterpolation}
I[f](\cdot):=\sum_{i\in\ZZ^{d} }f_i\beta_i(\cdot) \hspace{0.2cm} \mbox{for } f \in B(\mathcal{G}_{\rho}).
\ee

We recall  two basic results about the interpolation operator  $I$ (see e.g. \cite{Ciarlet,quartesaccosaleri07}). Given $\phi \in C_{b}(\RR^{d})$ (the space of bounded continuous functions on $\RR^{d}$), let us define $\hat{\phi} \in B(\mathcal{G}_{\rho})$ by $\hat{\phi}_i:=\phi(x_i)$ for all $i \in \ZZ^{d}$. 
 Suppose that $\phi: \RR^{d} \to \RR$ is Lipschitz with constant $L$. Then,  
\be\label{lipchitzianidad} I[\hat{\phi}] \; \; \; \mbox{is Lipschitz with constant $\sqrt{d} L$}.\ee 
On the other hand,
if $\phi \in \C^{2}(\RR^{d})$, with bounded second derivatives, then there exists $c>0$ such that 
\begin{equation}\label{www2}
\sup_{x\in \RR^{d}}| I[\hat{\phi}](x)-\phi(x)|=c \rho^2.
\end{equation}
\section{A fully discrete semi-Lagrangian scheme for the Hamilton-Jacobi Bellman equation}\label{mamdmoootroototootjkask}
Given $\mu \in C([0,T]; \P_1(\RR^{d}))$,  let us consider the equation  
\be\label{wewewssssaa}
\ba{rcl}
-\partial_{t} v  -     \half \mbox{tr}\left( \sigma(t) \sigma(t)^{\top} D^{2} v \right)  + \half | D v|^{2}   &=& F(x, \mu(t)) \;  \;    \hbox{in } \RR^{d} \times ]0,T[, \\[6pt]
v(x,T)&=& G(x, \mu(T)) \; \;   \mbox{for } x \in \RR^d.
\ea\ee
We discuss now a probabilistic interpretation  of  \eqref{wewewssssaa}.  
%%\cite[Theorem ??]{FleSon92}
Consider a probability space $(\Omega, \F, \PP)$, a filtration $\{\F_{t} \; ; \; t \in [0,T] \} $ and  a Brownian motion $W(\cdot)$  adapted to $\mathbb{F}:=\{\F_{s}^{t}\}_{s\in [t,T]}$. Define the space 
$$ L^{2,2}_{\mathbb{F}}:= \{ v \in L^{2}( \Omega \times [0,T]; \PP \otimes \dd t); \; \; v \; \; \mbox{is progressively measurable w.r.t. $\mathbb{F}$}\},$$
where $\dd t$ is the Lebesgue measure in $[0,T]$. For every $\alpha \in L^{2,2}_{\mathbb{F}}$,  set 
$$X^{x,t}[\alpha](s)= x-\int_{t}^{s}  \alpha(r) \dd r + \int_{t}^{s} \sigma(r) \dd W(r) \hspace{0.5cm} \forall \; s\in [t,T].$$
 Then, setting 
 \be\label{madmnnrrrpspsps}  v[\mu](x,t):= \inf_{\alpha \in L^{2,2}_{\mathbb{F}}}  \EE\left( \int_{t}^{T}\left[  \half |\alpha(s)|^{2}+ F(X^{x,t}[\alpha](s), \mu(s))\right] \dd s + G(X^{x,t}[\alpha](T), \mu(T))\right), \ee 
under {\bf(A1)},  classical  arguments (see \cite[Proposition 3.1 and Proposition 4.5]{YongZhou}) imply the existence of $  c>0$  such that  \small
\begin{eqnarray}\label{qmemnndnnnn}  \left|v[\mu](x, t)-v[\mu](x', t')\right| \leq  c\left[ |x-x'| + (1+ |x|\vee |x'|)\sqrt{|t'-t|}\right] \hspace{0.5cm} \forall \; x, x' \in \RR^{d}, \; \; t, t'\in [0,T],\\[4pt]
\label{qmewemnndnnnnq}  v[\mu](x+x', t)- 2v[\mu](x, t)+ v[\mu](x-x', t) \leq  c |x'|^{2}  \hspace{0.5cm} \forall \; x, x' \in \RR^{d}, \;  0\leq t \leq T.\end{eqnarray} \normalsize 
 Moreover,   by
%\begin{remark} {\rm(i)} Expression \eqref{madmnnrrrpspsps} provides a representation formula of a of a viscosity solution of  \eqref{wewewssssaa} in terms of an stochastic optimal control problem. Inequality \eqref{qmemnndnnnn} is a consequence of assumption {\bf(A1)}, \eqref{mwmrnnnsssasasadssdssdad} and classical estimates for solution of SDEs. The  semiconcavity property in  \eqref{qmewemnndnnnnq} follows easily from {\bf(A1)} and \eqref{mqmwmqemnnnsnsba}. For rigorous treatment of stochastic optimal control problems and the analysis of the value function $v[\mu]$ we refer the reader e.g. to the monographs \cite{FleSon92,YongZhou} and to the article \cite{}
%\end{remark}
 the continuity property implied by \eqref{qmemnndnnnn}, we can write directly the following dynamic programing principle for $v[\mu](\cdot, \cdot)$ (see e.g. \cite{doi:10.1137/090752328}):
\be\label{ecppd} v[\mu](x,t) = \inf_{\alpha \in L^{2,2}_{\mathbb{F}}}  \EE\left( \int_{t}^{t+h}\left[   \half |\alpha(s)|^{2}+ F(X^{x,t}[\alpha](s), \mu(s))\right] \dd s + v( X^{x,t}[\alpha](t+h), t+h) \right),
\ee
for all $h \in [0,T-t]$. Using \eqref{ecppd}  it is shown (see e.g. \cite[Theorem 3.1]{doi:10.1137/S0363012904440897}) that $v[\mu](x,t)$ is the unique viscosity solution of \eqref{wewewssssaa}.

  Given $\rho$, $h>0$ and $N$ such that $Nh=T$, expression  \eqref{ecppd} naturally induces the following scheme to solve \eqref{wewewssssaa}
\begin{equation}\label{dwoewoeoweaap}
\begin{cases}
  v_{i,k}=\hat{S}_{\rho, h}[\mu](v_{\cdot,k+1},i,k)    & \mbox{for all } i \in \mathcal{G}_{\rho},   \;  k=0, \hdots, N-1,\\
   v_{i,N}= G(x_{i}, \mu(t_N)), &  \mbox{for all } i \in \mathcal{G}_{\rho},
   \end{cases}
\end{equation}
where $\hat{S}_{\rho,h}[\mu]: B(\mathcal{G}_{\rho})\times \ZZ^{d} \times \{0,\hdots, N-1\} \to \RR$ is defined as

\be\label{eeqqqpapapwwwaappaa}\footnotesize
\ba{ll}
\hat{S}_{\rho,h}[\mu](f,i,k):=&\inf_{\alpha\in \RR^{d}} \left[\frac{1}{2r} \sum_{\ell=1}^{r} \left(I [f](x_{i}-h\alpha+\sqrt{h r}\sigma_{\ell}(t_{k}))+  I[f]( x_{i}-h\alpha-\sqrt{h r}\sigma_{\ell}(t_{k}))\right)\right. \\[6pt] 
 \; &\hspace{1.2cm}\left. +   \half h|\alpha|^{2}+hF(x_{i}, \mu(t_{k})) \right].\ea \ee
\normalsize

This scheme has been proposed in \cite{CamFal95} for a stationary second order possibly degenerate Hamilton-Jacobi-Bellman equation, corresponding to an infinite horizon stochastic optimal control problem. We now prove, in our evolutive framework, some basic properties of  $\hat{S}_{\rho,h}[\mu]$.
\begin{proposition}\label{propiedadesbasicasesquema1} The following assertions hold true: \\
{\rm(i)}  Suppose that $I[f]$ is Lipchitz with constant $L>0$. Then, there exists a compact set $K_{L}\subseteq \RR^{d}$ (whose diameter depends only on $L$)   such that the infima in the  r.h.s. of \eqref{eeqqqpapapwwwaappaa}  is attained in  the interior of $K_L$. \\
{\rm(ii)} For all $v, w \in  B(\mathcal{G}_{\rho})$ with $v\leq w$, we have that
$$ \hat{S}_{\rho,h}[\mu](v,i,k) \leq \hat{S}_{\rho,h}[\mu](w,i,k) \; \;  \mbox{for all } i \in \mathcal{G}_{\rho},   \;  k=0, \hdots, N-1.$$
{\rm(iii)} For every $c \in \RR$ and  $w \in  B(\mathcal{G}_{\rho})$ we have 
$$ \hat{S}_{\rho,h}[\mu](w+c,i,k)= \hat{S}_{\rho,h}[\mu](w,i,k)+c,  \; \;  \mbox{for all } i \in \mathcal{G}_{\rho},   \;  k=0, \hdots, N-1.$$
{\rm(iv)} Let  $(\rho_{n}, h_{n})\to 0$ {\rm(}as $n \uparrow \infty${\rm)} with $\rho_n^2=o({h_n})$ and consider a sequence of grid points $(x_{i_{n}}, t_{k_n}) \to (x,t)$ and a sequence $\mu_{n} \in C([0,T]; \P_{1}(\RR^{d}))$ such that $\mu_{n}\to \mu$. Then, for every 
$\phi \in  C_{c}^{\infty}\left(\RR^d  \times [0,T)\right)$, we have\footnotesize
$$
\lim_{n\to \infty}  \frac{1}{h_n} \left[\phi(x_{i_n},t_{k_n})-\hat{S}_{\rho_{n},h_{n}}[\mu_{n}](\phi_{k_{n+1}},i_n,k_n)\right]  
=-\partial_t \phi(x,t)-\half{\mbox{{\rm tr}}}\left( \sigma(t)\sigma(t)^{\top}D^2\phi(x,t)\right) +\half |D v|^{2}-F(x,\mu(t)),$$ \normalsize
where $\phi_{k}=\{\phi(x_i,t_k)\}_{i\in\ZZ ^d}$.
\end{proposition} \smallskip
\begin{proof} Properties {\rm(ii)} and {\rm(iii)} follows directly from \eqref{eeqqqpapapwwwaappaa}. Now, since $I[f]$ is bounded and continuous we directly obtain the existence of a minimizer  $\bar{\alpha}$ of the r.h.s. of \eqref{eeqqqpapapwwwaappaa}. Letting
 $$g(\alpha):= \frac{1}{2r} \sum_{\ell=1}^{r} \left(I [f](x_{i}-h\alpha+\sqrt{h  r}\sigma_{\ell}(t_{k}))+  I[f]( x_{i}-h\alpha-\sqrt{h r}\sigma_{\ell}(t_{k}))\right)$$
we have that $g$ is Lipschitz with constant $h\sqrt{d}L$ and
$$  \half h|\bar{\alpha}|^{2}  \leq g(0)-  g(\bar{\alpha}) \leq h\sqrt{d}L |\bar{\alpha}|. $$
The above expression implies that $|\bar{\alpha}| \leq 2\sqrt{d}L$, which proves {\rm(i)}. Now,  in order to prove  {\rm(iv)} let
$\phi\in C_{c}^{\infty}(\RR^{d})$ and notice that since $I[\phi(\cdot,t)]$ is Lipschitz with a constant depending only on $\| D \phi(\cdot, t)\|_{\infty}$ (and thus independent of $(\mu, \rho, h)$), we obtain by {\rm(i)} a fixed compact $K_{\phi}\subseteq \RR^{d}$ (depending only on $\phi$) such that the infima in the r.h.s. of \eqref{eeqqqpapapwwwaappaa} are attained in $K_{\phi}$.   Using this fact, for every $\ell=1,\hdots,r$ and $\alpha \in K_{\phi}$ a Taylor expansion yields to \small
 \be\label{T3}\ba{ll}
\phi(x_{i_n}-h_n\alpha+ \sqrt{h_n r}\sigma_{\ell}(t_{k_n}),t_{k_n+1})=& \phi(x_{i_{n}},t_{k_n+1})+ D \phi(x_{i_n},t_{k_n+1})^{\top}\left(-h_n\alpha + \sqrt{h_n r}\sigma_{\ell}(t_{k_n})\right)\\[4pt]
\; & + \frac{h_n r}{2}\sigma_{\ell}(t_{k_n})^{\top}D^{2} \phi(x_{i_n},t_{k_n+1})\sigma_{\ell}(t_{k_n})+ o(h_n),\\[4pt]
\phi(x_{i_n}-h_n\alpha - \sqrt{h_n r}\sigma_{\ell}(t_{k_n}),t_{k_n+1})=& \phi(x_{i_{n}},t_{k_n+1})+ D  \phi(x_{i_n},t_{k_n+1})^{\top}\left(-h_n\alpha - \sqrt{h_n r}\sigma_{\ell}(t_{k_n})\right)\\[4pt]
\; & + \frac{h_nr}{2}\sigma_{\ell}(t_{k_n})^{\top}D^{2} \phi(x_{i_n},t_{k_n+1})\sigma_{\ell}(t_{k_n})+ o(h_n).\ea
\ee
% \be
% \ba{rcl}\
%\phi(x_i-h\alpha\pm\sqrt{hd}\sigma_{\ell}(t))=\phi(x_i)+(-h\alpha\pm\sqrt{hd}\sigma_{\ell}(t))^{\top}D \phi(x_i)\\[4pt]
%+\half (-h\alpha\pm\sqrt{hd}\sigma_{\ell}(t))^{\top}D^2\phi(x_{i,l}^{\pm}(h,\alpha))(-h\alpha\pm\sqrt{hd}\sigma_{\ell}(t))
%%\pm T^3(-h\alpha\pm\sqrt{hd}\sigma_{\ell}(t),\phi)+O(h^2(1+\alpha+\alpha^2h^\frac{1}{2})),\nonumber
% \ea
% \ee
%where  $x_{i,l}^\pm(h,\alpha)=x_i+\xi_h^\pm (\a)(\pm h\alpha\pm\sqrt{hd}\sigma_{\ell}(t))$ with $\xi_h^\pm (\a)\in (0,1)$.
%$T^3(-h\alpha\pm\sqrt{hd}\sigma_{\ell}(t),\phi)$  we have indicate the third order term in Taylor expression.
%and where we have neglected the terms  smaller than $O(h^2(1+\alpha))$.\\
%Let us use the equality in \eqref{T3} in the following sum:
% \be \ba{rcl}\label{cons1}
% \frac{1}{2d} \sum_{\ell=1}^{d}\phi(x-h\alpha\pm\sqrt{hd}\sigma_{\ell}(t))=\phi(x)-h\alpha D \phi(x) + 
% \frac{h}{2} \sum_{\ell=1}^{d}\sigma_{\ell}(t))^{\top}D^2\phi(x_i) \sigma_{\ell}(t)\\
% \phi(x)-h\alpha D \phi(x) + 
%  \frac{h}{2}\sum_{\ell=1}^{d}\sum_{i,j=1}^{n}\sigma_{i,\ell}(t)\sigma_{k,\ell}(t)\partial_{x_i}\partial_{x_k}\phi(x)
%  \ea
%\ee
%Note that all odd terms (and in particular the terms in $h$ and $h^{3/2}$) cancel due to the symmetry of the two points $x-h\alpha\pm\sqrt{hd}\sigma_{\ell}(t)$.
%We observe that
%\be\label{idtrace}
%\sum_{\ell=1}^{d}\sum_{i,j=1}^{n}\sigma_{i,\ell}(t)\sigma_{k,\ell}(t)\partial_{x_i}\partial_{x_k}\phi(x)={\mbox {tr}}\left( \sigma(t)\sigma(t)^{\top}D^2\phi(x)\right)
%\ee\
\normalsize
Using  the interpolation error estimate \eqref{www2} and adding the equations in \eqref{T3}, we get \small
\be\ba{ll}\label{cons}
\phi(x_{i_n},t_{k_n})-\hat{S}_{\rho_{n},h_{n}}[\mu_{n}](\phi_{k_{n+1}},i_n,k_n) =&\phi(x_{i_{n}},t_{k_n})-\phi(x_{i_{n}},t_{k_n+1})-h_n F(x_{i_n},\mu_n(t_{k_n}))\\[4pt]
\;& -\frac{h_n}{2} \mbox{tr}(\sigma(t_{k_n}) \sigma(t_{k_n})^{\top} D^{2}\phi(x_{i_n}, t_{k_{n}+1}))\\[4pt]
\;& - h_{n} \inf_{\alpha\in \mbox{int}(K_{\phi})} \left[-D \phi(x_{i_{n}},t_{k_n+1})^{\top}\alpha +  \half h_n|\alpha|^{2} \right]\\[4pt]
\;&+ O(\rho_{n}^2)+ o(h_n).
% \inf_{\alpha\in \RR^{n}} \left[-h\alpha D \phi(x_i,t_{k_n}) +\half h |\alpha|^{2} -\sum_{\ell=1}^d \half (-h\alpha+\sqrt{hd}\sigma_{\ell}(t_{k_n})^{\top}D^2\phi(x^+_{i,l}(h,\alpha))(-h\alpha+\sqrt{hd}\sigma_{\ell}(t))\right.\\
%\left. -\sum_{\ell=1}^d \half (-h\alpha-\sqrt{hd}\sigma_{\ell}(t))^{\top}D^2\phi(x^-_{i,l}(h,\alpha))(-h\alpha-\sqrt{hd}\sigma_{\ell}(t_{k_n})\right]+O(\rho_n^2)\\
\ea \ee\normalsize
If we choose  $K_{\phi}$ large enough  such that for all $(x',t')\in \RR^{d} \times [0,T]$,  \small
$$\inf_{\alpha\in \mbox{int}(K_{\phi})} \left[-D \phi(x',t')^{\top}\alpha + \half |\alpha|^{2} \right]=\inf_{\alpha\in\RR^{d}} \left[-D \phi(x',t')^{\top}\alpha + \half |\alpha|^{2}  \right]=- \half |D\phi(x',t')|^{2},$$ \normalsize
then, dividing by $h_n$ and letting $h_{n}\downarrow 0$, we can pass to the limit in  \eqref{cons} to obtain the result.
\end{proof} \smallskip

We now define 
\be\label{mmfnnfnnfnndks}v_{\rho,h}[\mu](x,t):= I[v_{\cdot, \left[ \frac{t}{h} \right]}](x) \hspace{0.5cm} \mbox{for all } \hspace{0.2cm} (x,t) \in \RR^{d}\times [0,T].
\ee
Note that taking $t=t'$ in \eqref{qmemnndnnnn}, we have that $v[\mu](\cdot,t)$ is Lipschitz. We now prove the corresponding result for $v_{\rho,h}[\mu](\cdot,t)$ as well as a   discrete version of \eqref{qmewemnndnnnnq}.
\begin{lemma}\label{propiedadesdelavmuhrhocontinua} For every $t\in [0,T]$, the following assertions hold true:\smallskip\\
{\rm(i) [Lipschitz property]} The function  $v_{\rho,h}[\mu](\cdot,t)$ is Lipschitz with constant independent of $(\rho,h,\mu,t)$.  \\[4pt]
{\rm(ii) [Discrete semiconcavity]} There exists   $c>0$ independent of $(\rho,h,\mu,t)$ such that  \small
\be\label{semiconcavidaddebilcondosf}  v_{\rho,h}[\mu](x_i+x_{j},t)  - 2v_{\rho,h}[\mu](x_i,t) +v_{\rho,h}[\mu](x_i-x_{j},t) \leq  c  |x_{j}|^2 \hspace{0.4cm} \forall \; x_i, \; x_j \in \G_{\rho} \; \; \mbox{and $t\in [0,T]$}.
\ee \normalsize
\end{lemma}\smallskip
\begin{proof}   Using that  $\beta_{m}(x_{i+j}+z)= \beta_{m-j}(x_{i}+z)$, for every $m$,$i$, $j \in \ZZ^{d}$ and $z\in \RR^{d}$,  for every $\alpha \in \RR^{d}$, $k=0, \hdots, N-1$ and $\ell=1,\hdots,r$, we have that 
\be\label{rnwrnnassasdddas}\ba{c}I [v_{\cdot,k+1}](x_{i+j}-h\alpha + \sqrt{rh}\sigma_{\ell}(t_{k})) -I [v_{\cdot,k+1}](x_{i}-h\alpha + \sqrt{rh}\sigma_{\ell}(t_{k}))\\[4pt]
= \sum_{m \in \ZZ^{d}}\beta_{m}(x_{i}-h\alpha + \sqrt{rh}\sigma_{\ell}(t_{k}))(v_{m+j,k+1}-v_{m,k+1}), \ea\ee
with an analogous equality for the difference 
$$I [v_{\cdot,k+1}](x_{i+j}-h\alpha - \sqrt{rh}\sigma_{\ell}(t_{k})) -I [v_{\cdot,k+1}](x_{i}-h\alpha - \sqrt{rh}\sigma_{\ell}(t_{k})).$$
Since $G(\cdot, \mu)$   is Lipschitz  by {\bf{A1(i)}}, with a constant $c$ independent of $\mu$,   \eqref{dwoewoeoweaap}-\eqref{eeqqqpapapwwwaappaa}   imply that   $|v_{m+j,N}- v_{m,N}| \leq c|x_{m+j}-x_{m}|=c|x_{i+j}-x_{i}|$ for all $m\in \ZZ^{d}$.  Therefore,  since $\sum_{m\in \ZZ^{d}} \beta_{m}(x)=1$ for all $x\in \RR^{d}$, we obtain with {\bf{A1(i)}},  \eqref{dwoewoeoweaap}-\eqref{eeqqqpapapwwwaappaa}  and \eqref{rnwrnnassasdddas} that
$$ |v_{i+j, N-1}-v_{i,N-1}| \leq   c(1+ h)|x_{i+j}-x_{i}|$$
Therefore, by a recursive argument using \eqref{rnwrnnassasdddas} we easily obtain that 
$$ |v_{i+j, k}-v_{i,k}| \leq  c(1+ Th)|x_{i+j}-x_{i}| \hspace{0.3cm} \mbox{for all $i$, $j \in \ZZ^{d}$ and $k=0,\hdots, N$},$$
and assertion {\rm(i)} follows from \eqref{mmfnnfnnfnndks} and \eqref{lipchitzianidad}. 
%
%
% Since $G(\cdot, \mu(T))$ is {\color{red}$c_{0}$}-Lipschitz  we have that  $I[G](\cdot, \mu(T))$ is {\color{red}$c_0$}-Lipschitz.  Using \eqref{lipchitzianidad}, \eqref{dwoewoeoweaap} and \eqref{eeqqqpapapwwwaappaa} we obtain that $v_{\rho,h}[\mu](\cdot,t_{N-1})$ is ${\color{red}c_{0}}(1+h)$-Lipschitz and by a recursive argument,  we get that $v_{\rho,h}[\mu](\cdot,t_k)$ is ${\color{red}c_{0}}(1+Nh)={\color{red}c_{0}}(1+T)$ for all $k=0,\hdots, N$ and so, by \eqref{mmfnnfnnfnndks},   $v_{\rho,h}[\mu](\cdot,t)$ is  ${\color{red}c_{0}}(1+T)$ for all $t\in [0,T]$.  
In order to prove the second assertion note that, since $G$ is semiconcave, the result is valid  for $v_{\cdot,N}$. Inductively, we  suppose the result for $t_{k+1}$, i.e.
\be\label{mrmwmrwooofofofofofo}
 v_{i+j, k+1} - 2 v_{i, k+1} +  v_{i-j, k+1}\leq c |x_{j}|^{2}, \hspace{0.4cm} \forall \; i, \; j \in \ZZ^{d},
\ee
 and we prove its validity    for $t_{k}$  ($k=0,\hdots, N-1$). Let us denote by $\alpha_{i,k}$ an optimal solution for the problem defining $\hat{S}_{\rho, h}[\mu](v_{\cdot,k+1},i,k)$. Then\scriptsize
\be\label{dmdmdmaassssdd}\ba{rcl}
v_{i+j,k} &\leq&\frac{1}{2r} \sum_{\ell=1}^{r} \left[ I [v_{\cdot,k+1}](x_{i+j}-h\alpha_{i,k}+ \sqrt{rh}\sigma_{\ell}(t_{k})) +\frac{1}{2} I[v_{\cdot,n+1}](x_{i+j}-h\alpha_{i,k}- \sqrt{rh}\sigma_{\ell}(t_k)) \right]+ \half h |\alpha_{i,k}|^{2}\\[4pt]
\; &\; &+ h F(x_{i+j}, \mu(t_{k})),\\[2pt]
v_{i-j,k} &\leq&\frac{1}{2r} \sum_{\ell=1}^{r} \left[ I [v_{\cdot,k+1}](x_{i-j}-h\alpha_{i,k}+ \sqrt{rh}\sigma_{\ell}(t_{k})) +\frac{1}{2} I[v_{\cdot,n+1}](x_{i-j}-h\alpha_{i,k}- \sqrt{rh}\sigma_{\ell}(t_k)) \right]+    \half h |\alpha_{i,k}|^{2}\\[4pt]
\; &\; &+ h F(x_{i-j}, \mu(t_{k})),\\[2pt]
v_{i,k} &=&\frac{1}{2r} \sum_{\ell=1}^{r} \left[ I[v_{\cdot,k+1}](x_{i}-h\alpha_{i,k}+ \sqrt{rh}\sigma_{\ell}(t_{k})) +\frac{1}{2} I[v_{\cdot,n+1}](x_{i}-h\alpha_{i,k}- \sqrt{rh}\sigma_{\ell}(t_k)) \right]+   \half h |\alpha_{i,k}|^{2}\\[4pt]
\; &\; &+ h F(x_{i}, \mu(t_{k})).\\[2pt]
%v_{i+j,n} &\leq& \frac{1}{2}  I [v_{\cdot,n+1}](x_{i+j}-h\alpha_{i,n}\pm \sqrt{h}\sigma)  +\half  h |\alpha_{k,i}|^{2}  + h F(x_{i+j}, \mu(t_{n}))\\[4pt]
%v_{i-j,n} &\leq&  \half I [v_{\cdot,n+1}](x_{i-j}-h\alpha_{i,n}\pm \sqrt{h}\sigma) + \half h |\alpha_{i,n}|^{2}  + h F(x_{i-j}, \mu(t_{n}))\\[4pt]
%v_{i,n} &=&\half     I [v_{\cdot,n+1}](x_{i}-h\alpha_{i,n}\pm \sqrt{h}\sigma) + \half h |\alpha_{i,n}|^{2} + h F(x_{i}, \mu(t_{n}))
\ea
\ee\normalsize
On the other hand,  we have that \scriptsize
$$\ba{c}  I [v_{\cdot,k+1}](x_{i+j}-h\alpha_{i,k}+ \sqrt{rh}\sigma_{\ell}(t_{k}))-2I[v_{\cdot,k+1}](x_{i}-h\alpha_{i,k}+ \sqrt{rh}\sigma_{\ell}(t_{k}))+  I [v_{\cdot,k+1}](x_{i-j}-h\alpha_{i,k}+ \sqrt{rh}\sigma_{\ell}(t_{k}))=\\[4pt]
\sum_{m\in \ZZ^{d}} \beta_{m} (x_{i}-h\alpha_{i,k}+\sqrt{rh}\sigma_{\ell}(t_{k}))\left[ v_{m+j, k+1} - 2 v_{m, k+1} +  v_{m-j, k+1}\right]\leq c   |x_j|^2,\ea.$$
\normalsize
where the last inequality follows from \eqref{mrmwmrwooofofofofofo}. Analogously, \scriptsize
$$  I [v_{\cdot,k+1}](x_{i+j}-h\alpha_{i,k}- \sqrt{rh}\sigma_{\ell}(t_{k}))-2I[v_{\cdot,k+1}](x_{i}-h\alpha_{i,k}- \sqrt{rh}\sigma_{\ell}(t_{k}))+  I [v_{\cdot,k+1}](x_{i-j}-h\alpha_{i,k}- \sqrt{rh}\sigma_{\ell}(t_{k}))\leq c |x_j|^2.$$\normalsize
Therefore, combining \eqref{dmdmdmaassssdd},   the semiconcavity of $F$  and the above inequalities, we obtain    
$$ v_{i+j,k}-2v_{i,k} +v_{i-j,k}\leq   c(1+ h) |x_j|^2.$$
%$$\ba{rcl} v_{i+j,n}-2v_{i,n} +v_{i-j,n} &\leq& h {\color{red}C_{F}} | x_{j}|^{2} + \half  I [v_{\cdot,n+1}](x_{i+j}-h\alpha_{i,n}\pm \sqrt{h} \sigma)- 2I [v_{\cdot,n+1}](x_{i}-h\alpha_{i,n} \pm \sqrt{h} \sigma)\\[4pt]
%							\; & \; &  + I [v_{\cdot,n+1}](x_{i-j}-h\alpha_{i,n}\pm \sqrt{h} \sigma)\ea $$
%\normalsize
%$$v_{i+j,n}-2v_{i,n} +v_{i-j,n}  \leq$$
In particular, for $n=N-1$, we get 
$$v_{i+j,N-1}-2v_{i,N-1} +v_{i-j,N-1}  \leq c(1+h)|x_j|^2$$
and by recurrence,  for all $k=0,\hdots, N$, 
$$v_{i+j,k}-2v_{i,k} +v_{i-j,k}  \leq c(1+T) |x_j|^2$$
from which the result follows. \end{proof}\smallskip

Now, we regularize $v_{\rho,h}$ in the space variable. Let $\eps>0$ and  $\phi \in C_{0}^{\infty}(\RR^{d})$, with $\phi\geq 0$ and $\int_{\RR^{d}}\phi(x) \dd x= 1$. Define $\phi_{\eps}(x):= \frac{1}{\eps^{d}} \phi(x/\eps)$ and set
\be\label{ammrmooroorotttpsps} v^{\eps}_{\rho,h}[\mu](\cdot,t):= \phi_{\eps} \ast v_{\rho,h}[\mu](\cdot,t) \hspace{0.4cm} \forall \; t\in [0,T].\ee
 Using that  $v^{\eps}_{\rho,h}[\mu](\cdot,t)$ is Lipschitz by Lemma \ref{propiedadesdelavmuhrhocontinua}{\rm(i)},   we easily check that there exists $\gamma>0$ (independent of  $(\eps,\rho,h,\mu,t)$) such that 
\be\label{aproximacionuniformedelaconvolucion}\ba{rcl} \| v^{\eps}_{\rho,h}[\mu](\cdot,\cdot)-v_{\rho,h}[\mu](\cdot,\cdot)\|_{\infty} & \leq & \gamma \eps, \\[6pt]
											 \| D^{\alpha}v^{\eps}_{\rho,h}[\mu](\cdot,\cdot)\|_{\infty} &\leq & c_{\alpha}\eps^{1-|\alpha|} \ea
\ee
where  $\alpha$ is a multiindex with $|\alpha|>0$ and $c_{\alpha}>0$ depends only on $\alpha$. We have the following results whose proofs are provided in \cite{CS12}.
%The following result is crucial in what follows. 
%As a matter of fact, from the discrete semiconcavity property \eqref{semiconcavidaddebilcondosf}, we will obtain a uniform semiconcavity result for $v_{\rho, h}^{\eps}(\cdot, t)$ if $\rho=O(\eps^2)$.
\begin{lemma}\label{propiedadesdelavmuhrhocontinuaconeps} For every $t\in [0,T]$ we have that:\smallskip\\
{\rm(i)}  The function  $v^{\eps}_{\rho,h}[\mu](\cdot,t)$ is Lipschitz with constant   $c$ independent of $(\rho,h,\mu,t)$.  \\[4pt]
{\rm(ii)} If $d=1$, then 
\be\label{propiedadalaachdou}  \left(D v_{\rho,h}^{\eps}(x_j,t_k)- D v_{\rho,h}^{\eps}(x_i,t_k)\right)(x_{j}-x_{i}) \leq c(x_{j}-x_{i})^2 \hspace{0.2cm} \forall \; k=0,\hdots,N.\ee
%There exists   ${\color{red}d_{1}}>0$  independent of $(\rho,h,\eps,\mu,t)$, such that for all $x$, $y\in \RR^{d}$ we have
%\be\label{mmamaoooroorpappaaaa} v^{\eps}_{\rho,h}[\mu](x+  y,t)  - 2v^{\eps}_{\rho,h}[\mu](x,t) +v^{\eps}_{\rho,h}[\mu](x-  y,t)  \leq d_1\left( |y|^2 + \frac{\rho^{2}}{\eps}\right).\ee
%Moreover,  there exists   ${\color{red}d_{1}}>0$  independent of $(\rho,h,\eps,\mu,t)$, such that
%\be\label{mamdmmmdewqqaaaassss} \langle D^{2}v^{\eps}_{\rho,h}[\mu](x,t)y, y \rangle \leq  {\color{red}d_1}\left(1 + \left(\frac{\rho}{\eps^2}\right)^2\right)|y|^{2} \hspace{0.5cm} \forall \; x,y \in \RR^{d}.\ee 
%\be\label{mamdmmmdaaaassss} \langle D^{2}v^{\eps}_{\rho,h}[\mu](x,t)y, y \rangle \leq d_1\left(1 + \frac{\rho^{2}}{\eps^3}\right)|y|^{2} \hspace{0.5cm} \forall \; x,y \in \RR^{d}.\ee 
\end{lemma}
\begin{proof} See  \cite[Lemma 3.4{\rm(i)} and Lemma 3.6]{CS12}.
\end{proof} \vspace{0.3cm}

The following convergence result holds true:
\begin{theorem}\label{covergencefullydiscreteconeps}  Let  $(\rho_{n}, h_{n},\eps_{n})\to  0$  be such that $ \frac{\rho_{n}^{2}}{h_{n}} \to 0$ and  $\small \rho_{n}=o(\eps_n)$\normalsize. Then, for every sequence $\mu_{n}\in C([0,T]; \P_{1})$ such that $\mu_{n}\to \mu$ in $C([0,T]; \P_{1})$, we have that $v^{\eps_{n}}_{\rho_{n},h_{n}}[\mu_{n}]\to v[\mu]$ uniformly over compact sets and $Dv^{\eps_{n}}_{\rho_{n},h_{n}}[\mu_{n}](x,t)\to Dv[\mu](x,t)$ at every $(x,t)$ such that $Dv[\mu](x,t)$ exists.
\end{theorem}
\begin{proof} Using the properties of the scheme proved in Proposition \ref{propiedadesbasicasesquema1}, the first assertion follows by classical arguments (see \cite{BS91} and \cite[Theorem 3.3]{CS12}). The second assertion is proved following the same lines of the proof of  \cite[Theorem 3.5]{CS12}, which uses the uniform discrete semi-concavity of $v^{\eps_{n}}_{\rho_{n},h_{n}}[\mu_{n}]$, proved in our case in Lemma \ref{propiedadesdelavmuhrhocontinua},  and  \cite[Lemma 4.3 and Remark 4.4]{AchdouCamilliCorrias11}.
\end{proof}

\section{The fully-discrete scheme for the Fokker-Planck equation}\label{mnnwnrnwaa}
%%%%%%%%%%%%CONVERGENZA FULLY DISCRETE
Given a compact set $\K\subseteq \RR^{d}$ let us define  the convex and compact set 
\be\label{mqmwmqmwmnnnrrnnr} 
\SS_{\K}:= \left\{ (m_i)_{i \in \ZZ^{d}} \; ; \; m_{i}\geq 0\; \; \forall \; i\in \ZZ^{d}, \; \; m_{i}=0 \; \; \mbox{if $i\rho \notin \K$} \; \; \mbox{and } \; \sum_{i\in \ZZ^{d}} m_i=1\right\}.
\ee
For $\rho>0$ and $i\in \ZZ^{d}$ we set $E_{i}:= [x_{i}^{1}- \half \rho, x_{i}^{1}+ \half \rho] \times \cdots  [x_{i}^{d}- \half \rho, x_{i}^{d}+ \half \rho]$ and for a given $\mu = \{\mu_{i,k} \; ; \; i\in \ZZ^{d}, \; k =0,\hdots, N\} \in \SS_{\K}^{N+1}$ we define for all $k=0,\hdots, N$ the measure $\tilde{\mu}(t_k) \in \P_{1}(\RR^{d})$ as
\be\label{nademnennnnenbrtbtvvvvv} \dd \tilde{\mu}(t_{k}):= \frac{1}{\rho^{d}} \sum_{i \in \ZZ^{d}}\mu_{i,k} \mathbb{I}_{E_i}(x) \dd x\ee
 and its extension to all $t\in [0,T]$ by
\be\label{mmsamnnnrnnrnrooaa}\tilde{\mu}(t):=\left(\frac{t_{k+1}-t}{h}\right)\mu(t_{k}) +\left(\frac{t-t_{k}}{h}\right) \mu(t_{k+1}) \hspace{0.4cm} \mbox{if } \; \; t \in [t_{k}, t_{k+1}]. \ee
By construction  $\tilde{\mu} \in C([0,T]; \P_1(\RR^{d}))$ and without danger of confusion we will still write $\mu$ for $\tilde{\mu}$. 
%Let us define $a_{\rho,h}^{\eps}[\mu]: \RR^{d}\times [0,T] \times \RR^{d} \to \RR^{d}$ by
%\be\label{ewenqennqenqbbssbbbbababbbasss}
%a_{\rho,h}^{\eps}[\mu](x,t):= \partial_{p} H(x,t,Dv_{\rho,h}^{\eps}[\mu](x,t)).
%\ee
Thus, given $\mu \in \SS^{N+1}_{\K}$ we can define $v[\mu](\cdot, \cdot)$ as in Section \ref{mamdmoootroototootjkask}. 
For $\eps>0$,  $i \in \ZZ^{d}$, $\ell=1,\hdots,r$ and  $k=0,\hdots, N-1$ let us set
\be\label{flujofullydiscretounpaso}\ba{rcl}
\Phi^{\eps,\ell,+}_{i,k}[\mu]&:=&x_{i}- h Dv_{\rho,h}^{\eps}[\mu](x_{i},t_{k})+  \sqrt{rh}  \sigma_{\ell}(t_{k}), \\[6pt]
\Phi^{\eps,\ell,-}_{i,k}[\mu]&:=&x_{i}- h  Dv_{\rho,h}^{\eps}[\mu](x_{i},t_{k})-  \sqrt{rh}  \sigma_{\ell}(t_{k}),\ea
\ee
% Given $e \in \{-1,1\}^{d}$, let us define 
%$$
%\Phi^{\eps,e}_{i,k}[\mu]:=x_{i}- h  Dv_{\rho,h}^{\eps}[\mu](x_{i},t_{k})+  \sqrt{h} \sum_{j=1}^{d}e_{j} \sigma_{j}(t_{k})$$
and define $m[\mu]=\{m_{i,k}[\mu] \; ; \; i\in \ZZ^{d}, \; k =0,\hdots, N\}$ recursively  as  
\be\label{eqoooaaaaaaaaaaa}\ba{rcl} m_{i,k+1}[\mu]   &:=&  \frac{1}{2r} \sum_{j\in \ZZ^{d}} \sum_{\ell=1}^{r} \left[\beta_{i}\left( \Phi^{\eps,\ell,+}_{j,k}[\mu]\right)+\beta_{i}\left( \Phi^{\eps,\ell,-}_{j,k}[\mu]\right)\right] m_{j,k}[\mu], \\[4pt]
m_{i,0}[\mu]&:=& \int_{E_{i}} m_{0}(x) \dd x. \ea\ee
 \begin{remark}\label{mwrmmwrmmtmtmtmtttt} There exists a compact set $K_{h}\subseteq \RR^{d}$ such that \small $m[\mu] \in \SS_{K_{h}}^{N+1}$\normalsize. In fact,   using that $m_{0}$ has a compact support and that $\sigma$ and \small $Dv_{\rho,h}^{\eps}[\mu](x_{i},t_{k})$\normalsize are uniformly bounded {\rm(}by Lemma \ref{propiedadesdelavmuhrhocontinuaconeps}{\rm(i))}  we have the existence of a constant $ c>0$ such  that  $m_{i,k}=0$ if  \small $\rho i \notin   B(0,  c/\sqrt{h})$\normalsize, for every $k=0,\hdots, N$. Moreover, \small
$$ \sum_{i\in \ZZ^{d}} m_{i,k+1}[\mu]=  \sum_{j\in \ZZ^{d}} \frac{1}{2r}  \sum_{\ell=1}^{r} \sum_{i\in \ZZ^{d}}  \left[\beta_{i}\left( \Phi^{\eps,\ell,+}_{j,k}[\mu]\right)+\beta_{i}\left( \Phi^{\eps,\ell,-}_{j,k}[\mu]\right)\right] m_{j,k}[\mu]=  \sum_{j\in \ZZ^{d}}m_{j,k}[\mu]=  \sum_{j\in \ZZ^{d}}m_{j,0}[\mu]=1,  $$\normalsize
which implies that the scheme is conservative.
\end{remark}
Associated to \eqref{eqoooaaaaaaaaaaa}  we  set $m_{\rho,h}^{\eps}[\mu]:= \widetilde{m[\mu]} \in C([0,T]; \P_{1}(\RR^{d}))$, defined through \eqref{mmsamnnnrnnrnrooaa}, and for all $k=0,\hdots, N$ we define the measure
\be
\label{masmamsmasmas} \hat{m}_{\rho,h}^{\eps}[\mu](\cdot, t_{k}) :=\sum_{i \in \ZZ^{d}}  m_{i,k} [\mu] \delta_{x_{i}}(\cdot).
\ee
%which is extended to an element in $C([0,T]; \P_{1}(\RR^{d}))$ by linear interpolation as in  \eqref{mmsamnnnrnnrnrooaa}. 
%\small
%\be\label{eoqoeoqeoqoeqoeqe}\ba{rcl}  \dd m_{\rho,h}^{\eps}[\mu](x,t) &:=&  \left(\frac{t_{k+1}-t}{h}\right) \dd m_{\rho,h}^{\eps}[\mu](x,t_{k}) +\left(\frac{t-t_{k}}{h}\right) \dd m_{\rho,h}^{\eps}[\mu](x,t_{k+1}) \hspace{0.4cm} \mbox{if } \; \; t \in [t_{k}, t_{k+1}], \\[6pt]
% \hat{m}_{\rho,h}^{\eps}[\mu](\cdot,t) &:=&  \left(\frac{t_{k+1}-t}{h}\right) \hat{m}_{\rho,h}^{\eps}[\mu](\cdot,t_{k}) +\left(\frac{t-t_{k}}{h}\right)  \hat{m}_{\rho,h}^{\eps}[\mu](\cdot,t_{k+1}) \hspace{0.4cm} \mbox{if } \; \; t \in [t_{k}, t_{k+1}]. \ea
%\ee 
%\normalsize
Clearly, $\{ \hat{m}_{\rho,h}^{\eps}[\mu](\cdot, t_{k}) \; ; \; k=0, \hdots, N\} \in\P_{1}(\RR^{d})^{N+1}$. The following simple remark will be very useful in the sequel.
\begin{remark}[Probabilistic interpretation]\label{asmamdamdaepprpr}
Let us define \small
%\be\label{mamdmooeooeoerr}\ba{rcl}\PP(X_{k+1}=x_{i}\; \big | \; X_{k}= x_{j}) &=& \frac{1}{2^{d}} \sum_{e \in \{-1,1\}^{d}} \beta_{i}\left( \Phi^{\eps,e}_{j,k}[\mu]\right), \hspace{0.4cm}  \forall \; k=0, \hdots, N-1, \\[4pt]
%\PP(X_{0}=x_{i} )&=&m_{i,0}^{\eps}[\mu].
%\ea\ee
\be\label{mamdmooeooeoerr}\ba{rcl}p_{j,i}^{(k)}&:=& \frac{1}{2r} \sum_{\ell=1}^{r} \left[\beta_{i}\left( \Phi^{\eps,\ell,+}_{j,k}[\mu]\right)+\beta_{i}\left( \Phi^{\eps,\ell,-}_{j,k}[\mu]\right)\right] , \hspace{0.4cm}  \forall \; k=0, \hdots, N-1, \\[4pt]
p_{i}^{(0)}&=&m_{i,0}[\mu].
\ea\ee\normalsize
By classical results in probability theory {\rm(}see e.g. {\rm \cite{Breiman})} the family $\{ p_{j,i}^{(k)} \; ; \; j, \; i \in \ZZ^{d}, \; k=0,\hdots, N-1\}$ together with $\{ p_{i}^{(0)} \; ; \; i \in  \ZZ^{d}\}$ allow to define   a probability space $(\Omega, \F, \PP)$ and a discrete Markov chain $(X_{k})_{0\leq k \leq N}$ taking values in $\ZZ^d$, such that its initial distribution is given by $(p_{i}^{(0)})_{i\in \ZZ^{d}}$, the transition probabilities are given by \eqref{mamdmooeooeoerr} and the law at time $t_k$ is given 
by $ \hat{m}_{\rho,h}^{\eps}[\mu](\cdot, t_k)$. That is,
$$ \PP(X_{0}=x_i)=p_{i}^{(0)}, \; \; \PP(X_{k+1}=x_{i}\; \big | \; X_{k}= x_{j})= p_{j,i}^{(k)} \; \; \; \mbox{and } \; \PP(X_{k}=x_i)= m_{i,k}[\mu].$$
\end{remark}\smallskip

We have the following relation between the  $m_{\rho,h}^{\eps}[\mu] $ and  $\hat{m}_{\rho,h}^{\eps}[\mu]$:
\begin{lemma}\label{mamampprororoororortllltt} There exists a constant  $c>0$ {\rm(}independent of $(\rho,h,\eps,\mu)$\rm{)}  such that for all $k=0,\hdots, N$
$$ \bar{d}_{1}\left( m_{\rho,h}^{\eps}[\mu](\cdot,t_k), \hat{m}_{\rho,h}^{\eps}[\mu](\cdot, t_k)\right) \leq  c \rho.$$
\end{lemma}
\begin{proof} Let $\phi \in C(\RR^{d})$ be 1-Lipschitz. Then, by definition, 
$$ \int_{\RR^{d}} \phi(x) \dd \left[ m_{\rho,h}^{\eps}[\mu](\cdot,t_k)- \hat{m}_{\rho,h}^{\eps}[\mu](\cdot, t_{k})\right](x)= \sum_{i \in \ZZ^{d}} m_{i,k} [\mu]\left[  \frac{1}{\rho^{d}}\int_{E_{i}} \phi(x) \dd x - \phi(x_{i}) \right].$$
Then, the result follows, since for all $i \in \ZZ^{d}$, 
$$\left|\frac{1}{\rho^{d}}\int_{E_{i}} \phi(x) \dd x - \phi(x_{i})\right| \leq \frac{1}{\rho^{d}} \int_{E_{i}} |x-x_{i}| \dd x \leq c \rho.$$
\end{proof} \smallskip 

The following result will be the key to prove a compactness property for $ m_{\rho,h}^{\eps}[\mu] $.
%\be\label{eoweowoooaaa} \ee
\normalsize
\begin{proposition}\label{Ledwewewmmaboundform_hro}Suppose that $ \rho=O(h)$. Then, there exists  a constant   $c >0$ {\rm(}independent of  $(\rho,h,\eps,\mu)${\rm)} such that  for all $0\leq s\leq t \leq T$, we have that
\be d_{1}\left(m^{\eps}_{\rho,h}[\mu](t),m^{\eps}_{\rho,h}[\mu](s)\right) \leq c\sqrt{t-s}.\ee
\end{proposition}
\begin{proof} Let us first show that for all $k$, $k'=0, \hdots, N$, with $k'\leq k$,  we have that
\begin{eqnarray}\label{mmdmmadmwwwff}d_{1}(\hat{m}^{\eps}_{\rho,h}[\mu](t_{k}),\hat{m}^{\eps}_{\rho,h}[\mu](t_{k'}) )&\leq& c \sqrt{(k-k') h}=c \sqrt{t_{k}- t_{k'}},\\[4pt]
\label{mmdsdssdmmadmwwwff}d_{1}(m^{\eps}_{\rho,h}[\mu](t_{k}),m^{\eps}_{\rho,h}[\mu](t_{k'}) )&\leq& c\sqrt{(k-k') h}=c \sqrt{t_{k}- t_{k'}}.
\end{eqnarray}
For notational simplicity we will suppose that $k'=0$ and we omit the dependence on $\mu$. Consider the Markov chain $X_{(\cdot)}$ defined in Remark \ref{asmamdamdaepprpr} and let $\gamma\in \P(\RR^{d}\times \RR^{d})$ be the joint law  $X_{k}$ and $X_0$. By definition of $\db_1$ we have  that
\be\label{mfdmmfwoorootttpppspspsa}  d_{1}(\hat{m}^{\eps}_{\rho,h} (t_{k}),\hat{m}^{\eps}_{\rho,h}(0) ) \leq  \EE_{\PP}\left( \left| X_{k}-X_{0}\right|  \right),\ee
where $\PP$ is the probability measure introduced in Remark \ref{asmamdamdaepprpr} and $\EE_{\PP}(Y)= \int_{\Omega} Y(\omega)\dd \PP(\omega)$, for all $Y: \Omega \to \RR$ which are $\F$ measurable.  We have that\small
\be\ba{ll}\EE_{\PP}\left( \left| X_{k}-X_{0}\right|\right)&=\sum_{i_0,\hdots, i_{k}} \left|\sum_{p=0}^{k-1}(x_{i_{p+1}}- x_{i_{p}})\right|p_{i_{k-1},i_{k}}^{(k-1)}p_{i_{k-2},i_{k-1}}^{(k-2)} \dots p_{i_{0},i_{1}}^{(0)} m_{i_{0},0} ,\\[4pt]
\; &=\sum_{i_0,\hdots, i_{k-1}}\sum_{i_{k}} \left|x_{i_{k}}- x_{i_{k-1}}+\sum_{p=0}^{k-2}(x_{i_{p+1}}- x_{i_{p}})\right|p_{i_{k-1},i_{k}}^{(k-1)}p_{i_{k-2},i_{k-1}}^{(k-2)} \dots p_{i_{0},i_{1}}^{(0)} m_{i_{0},0} ,  \ea\ee \normalsize
and by  \eqref{mamdmooeooeoerr} we obtain  \footnotesize 
$$\ba{l} \sum_{i_{k}} \left|x_{i_{k}}- x_{i_{k-1}}+\sum_{p=0}^{k-2}(x_{i_{p+1}}- x_{i_{p}})\right|p_{i_{k-1},i_{k}}^{(k-1)}= \\ 
 \frac{1}{2r} \sum_{\ell=1}^{r}\sum_{i_{k}} \left|x_{i_{k}}- x_{i_{k-1}}+\sum_{p=0}^{k-2}(x_{i_{p+1}}- x_{i_{p}})\right|  \left[\beta_{i_k}\left( \Phi^{\eps,\ell,+}_{i_{k-1},k-1} \right)+\beta_{i_{k}}\left( \Phi^{\eps,\ell,-}_{i_{k-1},k-1} \right)\right].\ea$$ \normalsize
Using that $\rho=O(h)$, for $\ell=1,\hdots, r$ we have that \tiny
$$ \ba{rcl} \sum_{i_{k}} \left|x_{i_{k}}- x_{i_{k-1}}+\sum_{p=0}^{k-2}(x_{i_{p+1}}- x_{i_{p}})\right|   \beta_{i_k}\left( \Phi^{\eps,\ell,+}_{i_{k-1},k-1} \right)&\leq & \left| \Phi^{\eps,\ell,+}_{i_{k-1},k-1}- x_{i_{k-1}}+\sum_{p=0}^{k-2}(x_{i_{p+1}}- x_{i_{p}})\right| +  O(\rho), \\[4pt]
\; & = & \left| - h Dv_{\rho,h}^{\eps}(x_{i_{k-1}},t_{k-1})+  \sqrt{rh}  \sigma_{\ell}(t_{k-1}) +\sum_{p=0}^{k-2}(x_{i_{p+1}}- x_{i_{p}})\right| \\[4pt]
\; & \; &+  O(\rho)\\[4pt]
\; & \leq & \left|  \sqrt{rh}  \sigma_{\ell}(t_{k-1}) +\sum_{p=0}^{k-2}(x_{i_{p+1}}- x_{i_{p}})\right| +ch.
\ea$$\normalsize
Analogously, \footnotesize
$$  \sum_{i_{k}} \left|x_{i_{k}}- x_{i_{k-1}}+\sum_{p=0}^{k-2}(x_{i_{p+1}}- x_{i_{p}})\right|   \beta_{i_k}\left( \Phi^{\eps,\ell,-}_{i_{k-1},k-1} \right) \leq \left| - \sqrt{hr}  \sigma_{\ell}(t_{k-1}) +\sum_{p=0}^{k-2}(x_{i_{p+1}}- x_{i_{p}})\right| +  ch.$$
\normalsize
Thus, \small
$$\ba{l} \sum_{i_{k}} \left|x_{i_{k}}- x_{i_{k-1}}+\sum_{p=0}^{k-2}(x_{i_{p+1}}- x_{i_{p}})\right|p_{i_{k-1},i_{k}}^{(k-1)}\leq \\[4pt]
\frac{1}{2r} \sum_{\ell_{k-1}=1}^{r}\sum_{e_{\ell_{k-1}} \in \{-1,1\}} \left|  \sqrt{rh}e_{\ell_{k-1}}  \sigma_{\ell_{k-1}}(t_{k-1}) +\sum_{p=0}^{k-2}(x_{i_{p+1}}- x_{i_{p}})\right| +  ch.\ea$$
\normalsize
Therefore, \footnotesize
$$\ba{ll}\EE_{\PP}\left( \left| X_{k}-X_{0}\right|\right) \leq& \frac{1}{2r} \ds  \sum_{\ell_{k-1}=1}^{r} \sum_{e_{\ell_{k-1}} \in \{-1,1\}} \sum_{i_0,\hdots, i_{k-1}} \left|  \sqrt{rh}e_{\ell_{k-1}}  \sigma_{\ell_{k-1}}(t_{k-1}) +\sum_{p=0}^{k-2}(x_{i_{p+1}}- x_{i_{p}})\right|p_{i_{k-2},i_{k-1}}^{(k-2)} \dots p_{i_{0},i_{1}}^{(0)} m_{i_{0},0}\\
\; & +  ch.\ea$$\normalsize
By a recursive argument, we get
\be\label{mrmwmnnnananrrrrtttsa} \EE_{\PP}\left( \left| X_{k}-X_{0}\right|\right) \leq   \frac{\sqrt{rh}}{(2r)^{k}} \sum_{\ell_{k-1}, \hdots, \ell_{0} \in \{1,\hdots, r\}} \sum_{e_{\ell_{k-1}}, \hdots, e_{\ell_{0}}  \in \{-1,1\}} \left|    \sum_{p=0}^{k-1}e_{\ell_{p}}  \sigma_{\ell_{p}}(t_{p})\right|  
 +  ckh. \ee
Now, consider  $k$ steps of a random walk  in  $\RR^{r}$, i.e. a sequence of independent random vectors $Z_{0}, \hdots, Z_{k}$ in $\RR^{r}$,  defined  in $(\Omega, \F, \PP)$,  satisfying that for all $0\leq p \leq k$\small
$$ \PP( Z_{p}^{\ell}=1)=  \PP( Z_{p}^{\ell}=-1)= \frac{1}{2r} \hspace{0.4cm} \mbox{for all $\ell=1, \hdots, r \; \; $ and } \;  \; \PP\left( \bigcup_{1\leq \ell_1 < \ell_2 \leq r } \{  Z_{p}^{\ell_1}\neq 0\} \cap \{  Z_{p}^{\ell_2}\neq 0\}  \right)=0.$$ \normalsize
Then, by the Cauchy-Schwarz inequality, \small
$$ \frac{1}{(2r)^{k}} \sum_{\ell_{k-1}, \hdots, \ell_{0} \in \{1,\hdots, r\}} \sum_{e_{\ell_{k-1}}, \hdots, e_{\ell_{0}}  \in \{-1,1\}} \left|    \sum_{p=0}^{k-1}e_{\ell_{p}}  \sigma_{\ell_{p}}(t_{p})\right| = \EE_{\PP} \left( \left|    \sum_{p=0}^{k-1}   \sigma(t_{p}) Z_{p}\right| \right)\leq \left[ \EE_{\PP} \left( \left|    \sum_{p=0}^{k-1}   \sigma(t_{p}) Z_{p}\right|^{2} \right)\right]^{\half}.    $$\normalsize
Since $\EE^{\PP}( Z_{p})=0$, by independence we easily get that 
$$  \EE_{\PP} \left(    \sum_{p=0}^{k-1} \left|   \sigma(t_{p}) Z_{p}\right|^{2} \right)=\frac{1}{r}   \sum_{p=0}^{k-1} \mbox{tr}(\sigma(t_{p}) \sigma(t_{p})^{\top}), $$
and since $\sigma$ is bounded, we have that 
$$  \frac{1}{(2r)^{k}} \sum_{\ell_{k-1}, \hdots, \ell_{0}} \sum_{e_{\ell_{k-1}}, \hdots, e_{\ell_{0}} } \left|    \sum_{p=0}^{k-1}e_{\ell_{p}}  \sigma_{\ell_{p}}(t_{p})\right| \leq c\sqrt{k}, $$
for some $c>0$. Thus, combining \eqref{mfdmmfwoorootttpppspspsa}, \eqref{mrmwmnnnananrrrrtttsa}  and the above inequality, we obtain  that
$$d_{1}(\hat{m}^{\eps}_{\rho,h}[\mu](t_{k}),\hat{m}^{\eps}_{\rho,h}[\mu](t_{k'}) )\leq c \sqrt{kh}+ c kh = O(\sqrt{kh}), $$
which proves  \eqref{mmdmmadmwwwff}. By the triangular inequality we get  \small
$$ \ba{rcl} d_{1}\left(m^{\eps}_{\rho,h}[\mu](t_k),m^{\eps}_{\rho,h}[\mu](t_{k'})\right) &\leq&  d_{1}\left(m^{\eps}_{\rho,h}[\mu](t_k),\hat{m}^{\eps}_{\rho,h}[\mu](t_k)\right) +  d_{1}\left(m^{\eps}_{\rho,h}[\mu](t_{k'}),\hat{m}^{\eps}_{\rho,h}[\mu](t_{k'})\right) \\
													\, & \; & + d_{1}\left(\hat{m}^{\eps}_{\rho,h}[\mu](t_{k}),\hat{m}^{\eps}_{\rho,h}[\mu](t_{k'})\right).\ea$$\normalsize
Since $\rho=O(h)$, we get by  Lemma \ref{mamampprororoororortllltt} and \eqref{mmdmmadmwwwff} that 
$$  d_{1}\left(m^{\eps}_{\rho,h}[\mu](t_k),m^{\eps}_{\rho,h}[\mu](t_{k'})\right)  =O(  \rho + \sqrt{(k-k')h}) \leq O(\sqrt{t_{k}-t_{k'}}),$$
which proves \eqref{mmdsdssdmmadmwwwff}. Now, suppose that $s \in (t_{k_1}, t_{k_1+1})$ and $t \in (t_{k_2}, t_{k_2+1})$, then by the triangular inequality \footnotesize
\be\label{mdmmeoooeeooerrr}d_{1}\left(m^{\eps}_{\rho,h} (t),m^{\eps}_{\rho,h} (s)\right)\leq d_{1}\left(m^{\eps}_{\rho,h} (t_{k_1+1}),m^{\eps}_{\rho,h} (s)\right)  + d_{1}\left(m^{\eps}_{\rho,h}(t_{k_1+1}),m^{\eps}_{\rho,h} (t_{k_2})\right)+ d_{1}\left(m^{\eps}_{\rho,h}(t_{k_2}),m^{\eps}_{\rho,h} (t)\right).\ee \normalsize
Now, by  \eqref{mmsamnnnrnnrnrooaa} and \eqref{mmdsdssdmmadmwwwff} \footnotesize
$$\ba{rcl}  d_{1}\left(m^{\eps}_{\rho,h} (t_{k_1+1}),m^{\eps}_{\rho,h} (s)\right) +   d_{1}\left(m^{\eps}_{\rho,h} (t_{k_2}),m^{\eps}_{\rho,h} (t)\right) &\leq& \frac{t_{k_{1}+1}- s}{h} d_{1}\left(m^{\eps}_{\rho,h} (t_{k_{1}+1}),m^{\eps}_{\rho,h} (t_{k_{1}})\right)\\[4pt]
\; & \; &+  \frac{t- t_{k_2}}{h} d_{1}\left(m^{\eps}_{\rho,h} (t_{k_2}),m^{\eps}_{\rho,h} (t_{k_2+1})\right)\\[4pt]
\; &\leq & c\left[\frac{t-t_{k_{2}}}{\sqrt{h}}+ \frac{t_{k_{1}+1}-s}{\sqrt{h}}\right] \ea$$\normalsize
If  $k_{1}+1 \neq k_{2}$ we have, since $t-t_{k_{2}}\leq h$ and $t_{k_{1}+1}-s\leq h$,
\be\label{msmmotototoasoosos} d_{1}\left(m^{\eps}_{\rho,h} (t_{k_1+1}),m^{\eps}_{\rho,h} (s)\right) +   d_{1}\left(m^{\eps}_{\rho,h} (t_{k_2}),m^{\eps}_{\rho,h} (t)\right)  =O(\sqrt{h})=  O(\sqrt{t_{k_{2}}-t_{k_{1}+1}})= O(\sqrt{t-s}).\ee
If $k_{1}+1=k_{2}$, we have that $t-s \leq 2 h$
\be\label{msmmotototoasooqwqqaasos} d_{1}\left(m^{\eps}_{\rho,h} (t_{k_1+1}),m^{\eps}_{\rho,h} (s)\right) +   d_{1}\left(m^{\eps}_{\rho,h} (t_{k_2}),m^{\eps}_{\rho,h} (t)\right)  = O(\frac{ t-s}{\sqrt{h}})= O(\sqrt{t-s}). \ee
Therefore, since in both cases we have $d_{1}\left(m^{\eps}_{\rho,h} (t_{k_2}),m^{\eps}_{\rho,h} (t_{k_1+1})\right)= O(\sqrt{t-s}) $, inequalities \eqref{mdmmeoooeeooerrr} and \eqref{msmmotototoasoosos}-\eqref{msmmotototoasooqwqqaasos}  imply that
$$ 
d_{1}\left(m^{\eps}_{\rho,h} (t),m^{\eps}_{\rho,h} (s)\right) = O (\sqrt{t-s}).$$
\end{proof}\smallskip\\
Now, let us prove some uniform bounds for $m^{\eps}_{\rho,h}[\mu](\cdot)$ in $\P_{2}(\RR^{d})$. 
\begin{proposition} \label{masmammadapppqpqwppa} 
%The following assertions hold true: \smallskip\\
%{\rm(i)}   For all $t\in [0,T]$,  $m^{\eps}_{\rho,h}[\mu](t)$    has a support in $B(0,d_4/\sqrt{h})$.\smallskip\\
If  $\rho =O(\sqrt{h})$, then there exists  $c>0$  {\rm(}independent of  $(\rho,h,\eps,\mu)${\rm)} such that 
\be\label{mameiriiririiriiriirir} \int_{\RR^{d}} |x|^{2}  \dd m_{\rho, h}^{\eps}[\mu](t) \leq  c \hspace{0.5cm} \forall \; t\in [0,T].\ee
\end{proposition}
\begin{proof} By notational convenience we omit the dependence on $\mu$.  For every $k=0,\hdots, N-1$ we have
$$\int_{\RR^{d}} |x|^2 \dd m_{\rho,h}^{\eps}(x,t_{k+1})= \sum_{i \in \ZZ^{d}} \frac{1}{\rho^{d}}  \int_{E_{i}} |x|^2 \dd x \: m _{i,k+1},$$
but \small
$$\ba{ll}\sum_{i \in \ZZ^{d}} \frac{1}{\rho^{d}}  \int_{E_{i}}| x|^2 \dd x \: m _{i,k+1}&=  \frac{1}{2r} \sum_{i \in \ZZ^{d}} \frac{1}{\rho^{d}} \int_{E_{i}} |x|^2 \dd x \sum_{j\in \ZZ^{d}}\sum_{\ell=1}^{r}\left[ \beta_{i}(\Phi_{j,k}^{\eps,\ell,+}) +\beta_{i}(\Phi_{j,k}^{\eps,\ell,-})\right] m_{j,k},\\[6pt] 
 \; &=  \frac{1}{2r} \sum_{j\in \ZZ^{d}}  m_{j,k} \sum_{\ell=1}^{r}\sum_{i\in \ZZ^{d}}\left[ \beta_{i}(\Phi_{j,k}^{\eps,\ell,+}) +\beta_{i}(\Phi_{j,k}^{\eps,\ell,-})\right] \frac{1}{\rho^{d}} \int_{E_{i}}| x|^{2} \dd x.\ea$$
%
%$$\ba{c} \sum_{i \in \ZZ^{d}} \frac{1}{\rho^{d}}  \int_{E_{i}} x^2 \dd x \: m _{i,k+1}=\\[6pt]
%								     \frac{1}{2r} \sum_{i \in \ZZ^{d}} \frac{1}{\rho^{d}} \int_{E_{i}} x^2 \dd x \sum_{j\in \ZZ^{d}}\sum_{\ell=1}^{r}\left[ \beta_{i}(\Phi_{j,k}^{\eps,\ell,+}) +\beta_{i}(\Phi_{j,k}^{\eps,\ell,-})\right] m_{j,k}= \\[6pt]
%								 \half \sum_{j\in \ZZ} m_{j,k} \sum_{i\in \ZZ} \beta_{i}(x_{j}-h Dv^{\eps}_{\rho,h}(x_{j},t_{k}) \pm \sqrt{h}\sigma) \frac{1}{\rho} \int_{E_{i}} x^{2} \dd x. \ea$$\normalsize
\normalsize
Now, by a simple Taylor expansion we easily prove that     for $\phi \in C^{2}(\ov{E}_{i})$ we have
\be\label{amdmadmamdssdsdsds} \left| \frac{1}{\rho} \int_{E_{i}}\phi(x) \dd x -\phi(x_{i}) \right| = O(\rho ^2),  \hspace{0.4cm} \forall i \in \ZZ^{d}.\ee
Thus, letting $\phi(x)=|x|^{2}$, we get
%{\color{blue} and by using \eqref{www2} with $\phi(x)=|x|^{2}$ and  ${\color{blue}\OO}=E_{j^+}\cup E_{j^-}$, where $E_{j^+}$ is the grid cell containing  
%$\Phi_{j,k}^{\eps,\ell,+}$ and   $E_{j^-}$ is the grid  cell containing  $\Phi_{j,k}^{\eps,\ell,-}$,
%we get} 								   
$$\ba{rcl}\int_{\RR^{d}} x^2 \dd m_{\rho,h}^{\eps}(x,t_{k+1})&=&
\frac{1}{2r} \sum_{j\in \ZZ^{d}} m_{j,k}\sum_{\ell=1}^{r}\left( I[\hat{\phi}] (\Phi_{j,k}^{\eps,\ell,+})+ I[\hat{\phi}] (\Phi_{j,k}^{\eps,\ell,-})\right)+O(\rho^2),
\\[6pt]
%&=&
%\frac{1}{2r} \sum_{j\in \ZZ^{d}}m_{j,k} \left(\sum_{\ell=1}^{r}\left[  |\Phi_{j,k}^{\eps,\ell,+}|^2+|\Phi_{j,k}^{\eps,\ell,-}|^2+\rho^2 C\sup_{E_{j^+}\cup{E_{j^-}} }|D^{2}\phi(x)| \right]\right)+O(\rho^2)
%\\[6pt]
&=&
\frac{1}{2r} \sum_{j\in \ZZ^{d}}m_{j,k} \left(\sum_{\ell=1}^{r}\left[  |\Phi_{j,k}^{\eps,\ell,+}|^2+|\Phi_{j,k}^{\eps,\ell,-}|^2
 \right]\right)+O(\rho^2) ,\ea
$$ 							
%with $C$ constant independent on $(\rho,h,\eps,\mu$. In the last equality we have used the fact that 	   $\sup_{E_{j^+}\cup{E_{j^-}} }|D^{2}\phi(x)|=\sup_{x\in \RR^{d}} |D^{2}\phi(x)|$, for $\phi=|x|^2$.
where the last equality follows from   \eqref{www2}. Therefore, we get \small					
$$ \int_{\RR} |x|^2 \dd m_{\rho,h}^{\eps}(x,t_{k+1})=   \frac{1}{r}\sum_{j\in \ZZ^{d}} m_{j,k} \sum_{\ell=1}^{r} \left[ |x_{j}|^2 - 2  h \langle Dv^{\eps}_{\rho,h}(x_{j},t_{k}), x_{j} \rangle + h^2 |Dv^{\eps}_{\rho,h}(x_{j},t_{k})|^2 + h |\sigma_{\ell}(t_{k})|^2 \right]+  O\left( {\rho}^2 \right).$$
\normalsize
Now, using  that    $|  \langle Dv^{\eps}_{\rho,h}(x_{j},t_{k}), x_{j} \rangle| \leq \half ( c + |x_{j}|^2)$, for some $c>0$,  and that $\sigma$ is uniformly bounded,  we obtain 
$$\ba{rcl} \int_{\RR} |x|^2 \dd m_{\rho,h}^{\eps}(x,t_{k+1})&\leq&  (1+h) \sum_{j\in \ZZ^{d}} m_{j,k} | x_{j}|^2 +  O\left(h+ \rho^2 \right),\\[6pt]
									\; &= &(1+h)\int_{\RR^{d}} |x|^2 \dd m_{\rho,h}^{\eps}(x,t_{k})+  O\left(h+ \rho^2 \right), \ea$$
where we have used again \eqref{amdmadmamdssdsdsds}. Setting   $A_k:=  \int_{\RR^{d}} |x|^2 \dd m_{\rho,h}^{\eps}(x,t_{k})$ we get that
$$A_{k+1} \leq (1+h) A_k+  c\left(h+ \rho^2 \right),$$
for some $c>0$. Therefore, inductively for all $k_{1} = 0, \hdots, k$,  \small
$$\ba{ll}A_{k+1} \leq (1+h)^{k+1-k_{1}} A_{k_1} + c(h+\rho^{2}) \sum_{\ell=0}^{k-k_1}(1+h)^{\ell}&\leq (1+h)^{k+1} A_{0} + c(h+\rho^{2}) \left[\frac{(1+h)^{k+1}-1}{h}\right],\\[4pt]
																		& \leq e^{T} ( A_{0} +c'( 1+ \frac{\rho^{2}}{h}) ).\ea   $$ \normalsize
for some $c'>0$. Since $\rho^{2}=O(h)$ we get \eqref{mameiriiririiriiriirir} for all $t_{k}=0,\hdots, N$ and by \eqref{mmsamnnnrnnrnrooaa} for all $t\in [0,T]$.		
% by the Cauchy-Schwarz and \eqref{amdmadmamdssdsdsds} we have that  
%$$ \sum_{j\in \ZZ} m_{j,k}  \langle Dv^{\eps}_{\rho,h}(x_{j},t_{k}), x_{j} \rangle \leq C \left( \sum_{j\in \ZZ} m_{j,k} |x_j|^{2} \right)^{\half}= C\left( \int_{\RR} |x|^{2} \dd m_{\rho,h}(\cdot, t_{k}) \right)^{\half}+ O\left( \rho^{1/2} / h^{1/4}\right).$$
%Let us set $A_k:=  \int_{\RR} x^2 \dd m_{\rho,h}(x,t_{k})$. Using \eqref{amdmadmamdssdsdsds} again and setting we have that 
%$$A_{k+1} = A_k+ C hA_{k}^{1/2}+  O\left( h+ h^{3/4} \rho^{1/2}  + \rho/ h^{\half} \right)\leq (1+Ch) \max\{A_{k},1\} +  O\left( h+ h^{3/4} \rho^{1/2}  + \rho/ h^{\half} \right).$$
%Since $\rho=O(h^{3/2})$ we get then $A_k=O(1).$
\end{proof} \smallskip 

Our aim now is to obtain when $d=1$ uniform $L^{\infty}$-bounds for $m_{\rho,h}^{\eps}[\mu]$.
%Since $v^{\eps}_{\rho,h}[\mu]$ satisfies the discrete semiconcavity property \eqref{semiconcavidaddebilcondosf},  by \cite[Lemma 3.6]{CS12} we have that 
%$Dv_{\rho,h}^{\eps}[\mu]$ satisfies 
%\be\label{mmasmammsooorororpppsdass}  \left(D v_{\rho,h}^{\eps}(x_j,t_k)- D v_{\rho,h}^{\eps}(x_i,t_k)\right)(x_{j}-x_{i}) \leq {\color{red}c_{6}}(x_{j}-x_{i})^2 \hspace{0.2cm} \forall \; \; k=0,\hdots,N, \; \; i, \; j \in \ZZ.\ee
%%%%%
%This implies the following result.
We remark that for $d=1$ it suffices to consider also $r=1$. In this case the notation
can be simplified, and the superscript $\ell$ will be suppressed.

\begin{lemma}\label{resultadotecnico} Suppose that $d=1$ and  consider a sequence of numbers $\rho_n, h_n, \eps_n$ converging to $0$. Then, there exists a constant  $c>0$  (independent of $(n, \mu)$ for $n$ large enough)   such that   
\be\label{dis:discretetraj}
\min\left\{ |\Phi^{\eps_{n}, +}_{i,k}[\mu]-\Phi^{\eps_{n},+}_{j,k}[\mu]|^{2},  |\Phi^{\eps_{n}, -}_{i,k}[\mu]-\Phi^{\eps_{n},-}_{j,k}[\mu]|^{2} \right\}\geq \left(1-ch_{n} \right)|x_i-x_j|^2,
\ee
for all $i, j \in \ZZ$, $k=0,\hdots, N-1$ . As a consequence, there exists a constant $c>0$ (independent of $(n, \mu)$)  such that  
\be\label{cotasdelnumero}\sum_{j\in\ZZ} \left[\beta_{i}\left( \Phi^{\eps_{n},+}_{j,k}[\mu]\right)+\beta_{i}\left( \Phi^{\eps_{n},-}_{j,k}[\mu]\right)\right] \leq 1+ch_{n}. \ee
\end{lemma}

\begin{proof}For the reader's convenience, we omit the $\mu$ argument. By \eqref{flujofullydiscretounpaso}  we  have that\small
$$
\ba{ll} |\Phi^{\eps_{n},+ }_{i,k}-\Phi^{\eps_{n},+}_{j,k}|^2&= \left|x_i-x_j -h \left[Dv_{\rho_{n},h_{n}}^{\eps_{n}} (x_i,t_{k})-Dv_{\rho_{n},h_{n}}^{\eps_{n}}(x_j,t_{k})\right] +  \sqrt{h_n}\sigma(t_{k})  - \sqrt{h_n}\sigma(t_k)\right|^2,\\[6pt]
										\      	&\geq  |x_i-x_j |^2   -2h_n \left( Dv_{\rho_n,h_n}^{\eps_n} (x_i,t_k)-Dv_{\rho_n,h_n}^{\eps_n} (x_j,t_k)\right)(x_i-x_j ),
\ea
$$\normalsize
which together with the condition   Lemma \ref{propiedadesdelavmuhrhocontinuaconeps}(\rm{ii})  yields to 
$$|\Phi^{\eps_n,+ }_{i,k}-\Phi^{\eps_n,+}_{j,k}|^2 \geq   \left(1-c h_n \right)|x_i-x_j|^2.$$
for some $c>0$. Since the same argument is valid for $\Phi^{\eps_n,- }_{i,k}$, we get \eqref{dis:discretetraj}.
Using \eqref{dis:discretetraj} and following the proof in \cite[Lemma 3.8]{CS12}, we obtain that for all $k=0,\hdots, N-1$ and $i\in \ZZ$
$$ \sum_{j\in\ZZ} \beta_{i}\left( \Phi^{\eps_n,+}_{j,k}[\mu]\right)\leq 1+ch_n, \hspace{0.3cm} \mbox{and } \;  \sum_{j\in\ZZ} \beta_{i}\left( \Phi^{\eps_n,-}_{j,k}[\mu]\right)\leq 1+c h_n  $$
for some $c>0$, which implies  \eqref{cotasdelnumero}. 
\end{proof}\\

As a consequence we obtain the following uniform bound: 
\begin{proposition}\label{masmammpppqpqllls}  Suppose that   $d=1$ and consider a sequence of positive numbers $(\rho_n, h_n, \eps_n) \to 0$. Then,  there exists a constant $ c>0$, independent of $(n,\mu)$   such that 
\be\label{dmadmamameee} \|m^{\eps_{n}}_{\rho_{n},h_{n}}[\mu](\cdot,t)\|_{\infty}\leq c.\ee
\end{proposition}
\begin{proof}  We have that for all $k=0,\hdots, N-1$ and $x\in E_{i}$\small
$$\ba{ll} m^{\eps_{n}}_{\rho_{n},h_{n}}[\mu](x,t_{k+1})= \frac{1}{\rho_{n}} m_{i,k+1}[\mu]&= \sum_{j\in\ZZ}\left[\beta_{i}\left( \Phi^{\eps_{n},+}_{j,k}[\mu]\right)+\beta_{i}\left( \Phi^{\eps_{n},-}_{j,k}[\mu]\right)\right]\frac{1}{\rho_{n}}m_{j,k}[\mu], \\[4pt]
					\; & =  \sum_{j\in\ZZ}\left[\beta_{i}\left( \Phi^{\eps_{n},+}_{j,k}[\mu]\right)+\beta_{i}\left( \Phi^{\eps_{n},-}_{j,k}[\mu]\right)\right] m^{\eps_{n}}_{\rho_{n},h_{n}}[\mu](x_j,t_{k}),\\[4pt]
					\; & \leq \|m^{\eps_{n}}_{\rho_{n},h_{n}}[\mu](\cdot,t_{k})\|_{\infty} (1+ch_{n}), \ea$$\normalsize
by   \eqref{cotasdelnumero}. Therefore, by recurrence
$$\|m^{\eps_{n}}_{\rho_{n},h_{n}}[\mu](\cdot,t_{k+1})\|_{\infty} \leq  (1+ch_{n})^{N} \|m_{0}\|_{\infty} \leq e^{cT}\|m_{0}\|_{\infty}.$$
If $t\in ]t_k, t_{k+1}[$, by \eqref{mmsamnnnrnnrnrooaa} we have the same bound for $\|m^{\eps_{n}}_{\rho_{n},h_{n}}[\mu](\cdot,t)\|_{\infty}$.		
\end{proof}
\section{The fully discrete SL approximation of the second order mean field game problem}\label{qmeqmenresultados}
Given positive numbers $\rho$, $h$ and $\eps$ let us consider the problem 
$$ \mbox{Find $\mu \in C([0,T]; \P_1)$ such that $m_{\rho,h}^{\eps}[\mu] = \mu$.} \eqno(MFG)_{\rho,h}^{\eps}$$
or equivalently, recalling \eqref{eqoooaaaaaaaaaaa} and Remark \ref{mwrmmwrmmtmtmtmtttt}, find $\mu \in \SS_{\K_{h}}^{N+1}$ such that 

\be\label{eqoqeqqrrsdsrrrooaaaaaaaaaaa}\ba{rcl} \mu_{i,k+1} &:=&  \frac{1}{2r} \sum_{j\in \ZZ^{d}} \sum_{\ell=1}^{r} \left[\beta_{i}\left( \Phi^{\eps,\ell,+}_{j,k}[\mu]\right)+\beta_{i}\left( \Phi^{\eps,\ell,-}_{j,k}[\mu]\right)\right] \mu_{j,k}, \\[4pt]
\mu_{i,0} &:=& \int_{E_{i}} m_{0}(x) \dd x. \ea\ee\smallskip

We have the following existence result:
\begin{theorem}\label{existenciamfgesto} Problem $(MFG)_{\rho,h}^{\eps}$ has at least one solution.
\end{theorem}
\begin{proof} Let $\{\mu_n\}_{n \in \NN}$ and $\mu \in   \SS_{K_{h}}^{N+1}$ such that $\mu_n \to \mu$. Then, as elements in $C([0, T]; \P_{1}(\RR^{n})$ (see \eqref{nademnennnnenbrtbtvvvvv}-\eqref{mmsamnnnrnnrnrooaa}) we have that $\sup_{t\in [0,T]} d_{1}(\mu_{n}(t), \mu(t)) \to 0$. Therefore, by assumption $(A_0)$ we have that $v_{\rho,h}^{\eps}[\mu_{n}] \to v_{\rho,h}^{\eps}[\mu]$ uniformly and therefore  $Dv_{\rho,h}^{\eps}[\mu_{n}] \to Dv_{\rho,h}^{\eps}[\mu]$ uniformly. This implies that the function $\mu \in \SS_{K_{h}}^{N+1} \to m[\mu]\in \SS_{K_{h}}^{N+1}$ defined by \eqref{eqoooaaaaaaaaaaa} is continuous and since  $\SS_{K_{h}}^{N+1}$ is a non-empty convex  compact set the result follows from Brouwer fixed point Theorem. 
\end{proof} \smallskip

Now we can prove our main result:
\begin{theorem}\label{resultadoprincipalfullydiscrete} Suppose that  $d=1$ and that (A1)-(A3) hold true.  Consider a sequence of positive numbers $\rho_{n}, h_{n}, \eps_{n}$ satisfying that $\rho_{n}= O(h_n)$ and that $h_{n}= o(\eps_{n}^{2})$. Let    $\{m^{n}\}_{n\in \mathbb{N}}$ be a sequence of solutions of \small $(MFG)_{\rho_n,h_n}^{\eps_n}$\normalsize. Then  any  limit point $\ov{m}$ in \small$C([0,T];\P_{1})$\normalsize of $m^{n}$ (there exists at least one) solves \small $(MFG)$\normalsize. Moreover,  $m^{n} \to \ov{m}$ in  $L^{\infty}\left(\RR \times [0,T] \right)$-weak-$\ast$\normalsize. In particular, if $(MFG)$ has a unique solution $m$, then  $m^{n}\to m$  in  \small$C([0,T];\P_{1})$\normalsize  and in  \small$L^{\infty}\left(\RR^{d}\times [0,T] \right)$-weak-$\ast$\normalsize.
\end{theorem}
\begin{proof} For notational convenience we will write  $v^{n}:= v_{\rho_{n},h_{n}}^{\eps_{n}}[m^{n}]$. By Propositions \ref{Ledwewewmmaboundform_hro}-\ref{masmammadapppqpqwppa},  Lemma \ref{nwrnnnndbbruua} and Ascoli Theorem, there exists    $\ov{m}\in C([0,T]; \P_{1})$ such that, except for some subsequence,   $m^{n}$   converge to $\ov{m}$ in $ C([0,T]; \P_{1})$.  Our aim is to prove that \small
\be\label{eqfinal}
 \int_{\RR} \phi(x) \dd \ov{m}(t)(x)=  \int_{\RR} \phi(x) \dd m_{0}(x)   +\int_{0}^{t}\int_{\RR} \left[ \half \sigma^{2}(s) D^{2}\phi(x) -   D \phi(x) Dv [\ov{m}](x,s) \right] \dd \ov{m} (s)(x)\dd s.\ee \normalsize
Given $t\in [0,T]$, let us set $t_{n}:=  \left[ \frac{t}{h_{n}} \right] h_{n}$. We have
\be\label{teofinec1} \int_{\RR} \phi(x) \dd m^{n}(t_{n}) = \int_{\RR} \phi(x) \dd m_{0}(x)+   \sum_{k=0}^{n-1}  \int_{\RR}\phi(x) \dd  \left[m^{n}(t_{k+1})-m^{n}(t_{k})\right].\ee
%By definitions \eqref{defmconmu} and \eqref{defmisuradepdemu}, setting  for all $k=0, \hdots,n-1$, 
%$$\ba{c}\Phi^{n,\ell,+}_{k}(x):=x- h_n  D v^{n}(x,t_k) +\sqrt{dh} \sigma_{\ell}(t_{k}) , \hspace{0.4cm}  \; \;  \Phi^{n,\ell-}_{k}(x):=x- h_n  D v^{n}(x,t_k) -\sqrt{dh} \sigma_{\ell}(t_{k})\\[4pt]
%\Phi^{n,\ell,+}_{i,k}:=\Phi^{n,\ell,+}_{k}(x_{i}) , \hspace{0.4cm}  \; \;  \Phi^{n,\ell,-}_{i,k}:=\Phi^{n,\ell,-}_{k}(x_{i})  \ea$$
By \eqref{nademnennnnenbrtbtvvvvv}-\eqref{eqoooaaaaaaaaaaa} and \eqref{www2}, we obtain 
\be\label{serieecsdemfinal}\ba{rcl}  \int_{\RR}\phi(x) \dd m^{n}(t_{k+1}) &=&\sum_{i\in  \ZZ} m^{n}_{i,k+1}  \frac{1}{\rho_{n} }  \int_{E_{i}} \phi(x) \dd x,\\[6pt]
									\; &=& \sum_{i\in  \ZZ} m^{n}_{i,k+1}  \phi(x_{i}) + O(\rho_{n}^{2}),\\[6pt]
									\     & =&  \sum_{i\in \ZZ} \phi(x_i) \sum_{j\in \ZZ}   m^{n}_{j,k}\left[  \beta_{i}\left(\Phi^{\eps_n,+}_{j,k}\right)   + \beta_{i}\left(\Phi^{\eps_n,+}_{j,k}\right)\right]+ O(\rho_{n}^{2}), \\[6pt]
									\ 	&=&  \sum_{j\in \ZZ} m^{n}_{j,k}\sum_{i\in \ZZ} \phi(x_{i}) \left[  \beta_{i}\left(\Phi^{\eps_n,+}_{j,k}\right)   + \beta_{i}\left(\Phi^{\eps_n,+}_{j,k}\right)\right]+ O( \rho_{n}^{2}),\\[6pt]
										\; &=&  \sum_{j\in \ZZ} m^{n}_{j,k}\left[ \phi\left(\Phi^{\eps_n,+}_{j,k}\right)   + \phi\left(\Phi^{\eps_n,+}_{j,k}\right)\right]+ O( \rho_{n}^{2}).  \ea \ee
Let us set \small
$$ \Phi^{\eps_n,+}_{k}(x) :=x- h_n  Dv^{n} (x,t_{k})+  \sqrt{h_n}  \sigma(t_{k}), \; \; \; \Phi^{\eps_n,-}_{k}(x) :=x - h_n  Dv^{n} (x,t_{k})-  \sqrt{h_n}  \sigma(t_{k}).$$ \normalsize
Taking $|\alpha|=3$ in  the second inequality of \eqref{aproximacionuniformedelaconvolucion} we easily obtain by a Taylor expansion that  
%$$ \left| \phi\left(\Phi^{\eps_n,\ell,+}_{k}(x)\right)-\phi\left(\Phi^{\eps_n,\ell,+}_{k}(y) \right)\right| \leq   {\color{red}c'} \left(1+\frac{h_n}{\eps_{n}}\right) |x-y|,$$ 
%with an analogous inequality for $\phi\left(\Phi^{\eps_n,\ell,-}_{k}(x) \right)$. Since  ${\color{blue} \frac{h_n}{\eps_{n}} =O(1)}$, we obtain  
\small
$$ \left|\frac{1}{\rho_n} \int_{E_{j}}  \phi\left(\Phi^{\eps_n,+}_{k}(x)\right) \dd x - \phi\left(\Phi_{j,k}^{\eps_n,+}\right) \right|+\left|\frac{1}{\rho_n} \int_{E_{j}}  \phi\left(\Phi^{\eps_n,-}_{k}(x)\right) \dd x - \phi\left(\Phi_{j,k}^{\eps_n,-}\right) \right| \leq c  h_n \frac{\rho_n^{2}}{\eps_n},$$\normalsize
%with an analogous inequality for $\phi\left(\Phi^{n,\ell,-}_{k}(\cdot) \right)$. 
%Therefore, combining with \eqref{serieecsdemfinal}, we get (recalling  \eqref{www} with $\gamma=1$)
%\be\label{seciguald}\ba{rcl} \int_{\RR}\phi(x) \dd m^{n}(t_{k+1})& =& \sum_{j\in \ZZ} m^{n}_{j,k} \sum_{i\in \ZZ} \beta_{i}\left(\Phi_{j,k,k+1}^{n} \right)\phi(x_{i}) + O(\rho_n), \\[6pt]
%									\ &= &   \sum_{j\in \ZZ} m^{n}_{j,k}  I [\phi] \left(\Phi_{j,k,k+1}^{n}\right) + O(\rho_n), \\[6pt]
%									\ &= &   \sum_{j\in \ZZ} m^{n}_{j,k}  \phi\left(\Phi_{j,k,k+1}^{n} \right) + O(\rho_n).\ea \ee
%Analogously, 
%$$  \int_{\RR}\phi(x) \dd m^{n}(t_{k})=\sum_{j\in \ZZ^{d}} m^{n}_{j,k}  \phi(x_{j}) + O(\rho).$$
%On the other hand, by Lemma \ref{propiedadesdelavmuhrhocontinua}{\rm(i)}, the function  $v^{n}(\cdot, t)$ is Lipschitz (with Lipschitz constant independent of $n$). Therefore, by \eqref{aproximacionuniformedelaconvolucion}   we have the existence of a constant $c>0$ (independent of $n$) such that 
%\be\label{stimaDv} \left| D v^{n}(x,t) - D v^{n}(y,t) \right| \leq \frac{c}{\eps_{n}} |x-y|,\ee
%which implies, setting $\Phi^{n}_{k,k+1}(x)=x- h_n D v^{n}(x,t)$, that
%$$ \left| \phi\left(\Phi^{n}_{k,k+1}(x) \right)-\phi\left(\Phi^{n}_{k,k+1}(y) \right)\right| \leq   c' \left(1+\frac{h_n}{\eps_{n}}\right) |x-y|.$$
%for some $c'>c$ (which is also independent of $n$).  Therefore, we have 
%
%
%Since ,  we get
for some $c>0$. Therefore, 
$$\ba{rcl} \int_{\RR}\phi(x) \dd m^{n}(t_{k+1})&=&   \sum_{j\in \ZZ}\int_{E_{j}}\left[  \phi\left(\Phi^{\eps_n,+}_{k}(x)\right)   + \phi\left(\Phi^{\eps_n,-}_{k}(x)\right)\right] \dd x \frac{m^{n}_{j,k}}{\rho_{n}} \\[4pt]
\  &&+ O( h_n \frac{\rho_n^{2}}{\eps_n}+\rho_{n}^{2}),\\[4pt]
									\  &= &   \int_{\RR}\left[  \phi\left(\Phi^{\eps_n,+}_{k}(x)\right)   + \phi\left(\Phi^{\eps_n,-}_{k}(x)\right)\right] \dd  m^{n}(t_{k}) \\[4pt]
									\  &\; &+  O\left( h_n \frac{\rho_n^{2}}{\eps_n}+\rho_n^{2}\right).\ea $$
By a Taylor expansion we find that \small
$$
\half\left[ \phi\left(\Phi^{\eps_n,+}_{k}(x)\right)+  \phi\left(\Phi^{\eps_n,+}_{k}(x)\right) \right]-\phi(x)= -h_{n} \left[  Dv^{n}(x,t_{k}) D\phi(x)  +\frac{1}{2}\sigma^{2} (t_{k}) D^{2}\phi(x)  \right] + O(h_{n}^{2}), 
$$\normalsize								
The expression above yields to\small
\be\label{ecsda}\ba{rcl} \int_{\RR}\phi(x) \dd  \left[m^{n}(t_{k+1})-m^{n}(t_{k})\right] &=&  - h_{n}\int_{\RR} \left[  Dv^{n}(x,t_{k}) D\phi(x)   + \half \sigma ^2 (t_k) D^{2}\phi(x) \right]\dd  m^{n}(t_{k}) \\[4pt]
												              \     & \ &  + O\left(h_{n}^{2}+ h_n \frac{\rho_n^{2}}{\eps_n}+\rho_{n}^{2}\right).\ea\ee \normalsize
 Since  by  the second inequality of \eqref{aproximacionuniformedelaconvolucion} the term inside the integral in \eqref{ecsda}  is  $ c/\eps_{n}$-Lipschitz (with  $c$ large enough) w.r.t. $x$, Proposition  \ref{Ledwewewmmaboundform_hro} gives that for all $s\in [t_{k}, t_{k+1}]$, with $k=0,\hdots, n-1$, we have
 $$ \left|\int_{\RR}    Dv^{n}(x,s)  D\phi(x)   \dd \left[ m^{n}(s)(x)-  m^{n}(t_{k})(x)\right]\right| \leq \frac{c}{\eps_{n}} \sqrt{s-t_k} \leq \frac{c\sqrt{h_{n}}}{\eps_{n}},$$
 which implies that, since $Dv^{n}(x,t_{k})=Dv^{n}(x, s)$ for all $s\in [t_k, t_{k+1}[$,  \small
\be\label{unaexp}
 \left| \int_{t_{k}}^{t_{k+1}}\int_{\RR}     Dv^{n}(x,s) D\phi(x)   \dd   m^{n}(s)(x) \dd s -  \int_{t_{k}}^{t_{k+1}}\int_{\RR}     Dv^{n}(x,t_k)  D\phi(x)    \dd   m^{n}(t_{k})(x) \dd s\right| \leq  \frac{ch_{n}^{\frac{3}{2}}}{\eps_{n}}.
\ee\normalsize
% Using that 
% $$ \left| \int_{t_{k}}^{t_{k+1}}\int_{\RR}    \mbox{Tr}(\sigma \sigma^{\top}(t_k) D^{2}\phi(x)) \dd   m^{n}(s)(x) \dd s -  \int_{t_{k}}^{t_{k+1}}\int_{\RR} \mbox{Tr}(\sigma \sigma^{\top}(t_k) D^{2}\phi(x))  \dd   m^{n}(t_k)  \dd s\right| \leq   {\color{red}c''}h_{n}^{\frac{3}{2}.}$$
% \be\label{unaexp} \ba{l}  \left| \int_{t_{k}}^{t_{k+1}}\int_{\RR}  \left[\langle D^{n}(x,s), D\phi(x)\rangle -\langle Dv^{n}(x,t_{k}), D\phi(x)\rangle\right]\dd \left[ m^{n}(s)(x)-  m^{n}(t_{k})(x)\right]\dd s  \right| \\
% + \left| \int_{t_{k}}^{t_{k+1}}\int_{\RR}  \left[\mbox{Tr}(\sigma \sigma^{\top}(s) D^{2}\phi(x))D\phi(x) +\mbox{Tr}(\sigma \sigma^{\top}(t_k) D^{2}\phi(x)) \right] \dd \left[ m^{n}(s)-  m^{n}(t_{k})\right]\dd s \right|\leq \frac{{\color{red}c''}h_{n}^{\frac{3}{2}}}{\eps_{n}}.\ea\ee \normalsize
 Therefore, combining  \eqref{ecsda} and  \eqref{unaexp}, we obtain that \small
 $$\ba{rcl}  \int_{\RR}\phi(x) \dd  \left[m^{n}(t_{k+1})-m^{n}(t_{k})\right]  &=& -\int_{t_{k}}^{t_{k+1}}\int_{\RR}   Dv^{n}(x,s) D\phi(x)  \dd  m^{n}(s)(x) \dd s   \\[4pt]
 											       \ &  \ &+h_{n}\int_{\RR} \half \sigma^2(t_k) D^{2}\phi(x) \dd  m^{n}(t_k)(x)\\[4pt]
											       	\; & \; & +O\left(h_{n}^{2}+ h_n \frac{\rho_n^{2}}{\eps_n}+  \frac{h_{n}^{\frac{3}{2}}}{\eps_{n}}+ \rho_{n}^{2}\right).\ea $$\normalsize
 Thus, summing from $k=0$ to $k=n-1$ and using  \eqref{teofinec1} \small
 \be\label{antesdepasarallimite}\ba{rcl}  \int_{\RR} \phi(x) \dd m^{n}(t_{n})(x)&=& \int_{\RR} \phi(x)  m^{n}(x,0)-\int_{0}^{t_{n}}\int_{\RR}  Dv^{n}(x,s) D\phi(x)   \dd  m^{n}(s)(x) \ \dd s \\[6pt]
 									\       & \ & +h_{n}\sum_{k=0}^{n-1}\int_{\RR} \half \sigma^{2}(t_k) D^{2}\phi(x) \dd  m^{n}(t_k)(x) + O\left( h_{n}+
									
									\frac{\rho_{n}}{\eps_{n}}+ \frac{\sqrt{h}}{\eps_{n}}+ \frac{\rho_{n}^{2}}{h_{n}}\right), \\[6pt]
									&=& \int_{\RR} \phi(x)  m^{n}(x,0)-\int_{0}^{t_{n}}\int_{\RR}  Dv^{n}(x,s) D\phi(x)   \dd  m^{n}(s)(x) \ \dd s \\[6pt]
 									\       & \ & +h_{n}\sum_{k=0}^{n-1}\int_{\RR} \half \sigma^{2}(t_k) D^{2}\phi(x)  \dd  \bar{m}(t_k)(x) \\[6pt]
									\; & \; &+ O\left( \sup_{s\in [0,T]} d_{1}(m_{n}(s), \bar{m}(s))+  h_{n}+
									
									\frac{\rho_{n}}{\eps_{n}}+ \frac{\sqrt{h}}{\eps_{n}}+ \frac{\rho_{n}^{2}}{h_{n}}\right).\ea\ee \normalsize								
%{\color{red} Let us prove {\rm(i)}}. By Theorem \ref{covergencefullydiscreteconeps},   for all $s\in [0,T]$ we have that $a^{n}(\cdot,s) \to a [\ov{m}](\cdot,s)$ uniformly in $\mbox{supp} \phi$. Therefore, since if a sequence of bounded functions $f_{n}$ converges uniformly to $f$ then  $\int_{\RR^{d}} f_{n}   \dd  m^{n}(s)(x) \to \int_{\RR^{d}} f   \dd  \bar{m}(s)(x)$, we can pass to the limit in the above relation to get  \eqref{eqfinal}.\medskip\\
Since $ t \in [0,T] \to \int_{\RR} \sigma^{2}(t) D^{2}\phi(x)  \dd  \bar{m}(t)(x)$
is continuous (because $\sigma$ is continuous and $\bar{m}\in C([0,T]; \P_{1})$), we have that
\be\label{nuevolimite}\lim_{n\to \infty} h_{n}\sum_{k=0}^{n-1}\int_{\RR} \half \sigma^{2}(t_k) D^{2}\phi(x)  \dd  \bar{m}(t_k)(x) = \int_{0}^{t} \int_{\RR} \half \sigma^{2}(s) D^{2}\phi(x) \dd  \bar{m}(s)(x) \dd s.\ee
Moreover,  Proposition \ref{masmammpppqpqllls} implies that the density of  $m^n$ (still denoted by $m^n$) is bounded in $L^{\infty}\left(\RR \times [0,T] \right)$. Thus,    $\ov{m}$ is absolutely continuous and $m_n \to \ov{m}$ in  $L^{\infty}\left(\RR \times [0,T] \right)$-weak-$\ast$. 
%We show now that $\ov{m}$ solves $(MFG)$,   i.e. for any $t\in [0,T]$ and $\phi \in C_{c}^{\infty}(\RR^{d})$.
%By Theorem \ref{covergencefullydiscreteconeps}  and the continuity of $\partial_{p} H(x,t, \cdot)$ we have that  $a^{n}(x,s) \to a[\ov{m}](x,s)$ for a.a. $(x,s) \in \RR\times [0,T]$. 
On the other hand, using that $\phi\in C^{\infty}_{c}(\RR)$, that for all $t\in [0,T]$ the derivative $Dv[\ov{m}](x,t)$ exists for a.a. $x$ (by \eqref{qmemnndnnnn}),  Theorem \ref{covergencefullydiscreteconeps}   and   the Lebesgue theorem, we get that 
\be\label{ultimolimite}  \mathbb{I}_{[0,t_{n}]}   Dv^{n}(\cdot,\cdot)   D\phi(\cdot)    \to   \mathbb{I}_{[0,t]}   Dv[\ov{m}](\cdot,\cdot)  D\phi(\cdot)     \hspace{0.3cm} \mbox{strongly in $L^1(\RR^{d}\times [0,T])$},\ee
Thus, since $m^{n}$ converge to $\ov{m}$ in $L^{\infty}\left(\RR\times [0,T] \right)$-weak-$\ast$, using \eqref{nuevolimite}-\eqref{ultimolimite}, that $\rho_{n}= O(h_{n})$ and that $h_{n}=o(\eps_{n}^{2})$, we can pass to the limit in \eqref{antesdepasarallimite} to obtain \eqref{eqfinal}.  %$$ m^{n}(x,s)-  m^{n}(x,t_{k})= \frac{s-t_{k}}{h} \left(  m^{n}(x,t_{k+1})-m^{n}(x,t_{k+1})   \right),$$ 
%expression \eqref{unaexp} can be estimated by 
% $$ \int_{t_{k}}^{t_{k+1}}  \frac{(s-t_{k})}{h} \left|\int_{\RR}  D\phi(x)\cdot Dv^{n}(x,t_{k})  \left[ m^{n}(x,t_{k+1})-  m^{n}(x,t_{k})\right] \right| \dd s.$$ 
% Since    $D\phi(\cdot)\cdot Dv^{n}(\cdot,t_{k})$ is $C/\eps$-Lipschitz (with $C$ big enough), Lemma ? gives
%												              
%Since $\|D\phi(\cdot)Dv^{n}(\cdot,t_{n})\|_{\infty}$ is uniformly bounded, we have that 			
%  $$ \int_{\RR}\phi(x) \dd  \left[m^{n}(t_{k+1})-m^{n}(t_{k})\right]= 	- h_{n}\int_{\RR} D\phi(x)Dv^{n}(x,t_{n}) \dd  m^{n}(t_{k}) + O\left(d_{1}(m^{n},\bar{m})+ h_{n}^{2}+\frac{\rho_{n}}{\eps_{n}}\right).$$
%On the other hand, since
% $$  		- h_{n}\int_{\RR} D\phi(x)Dv^{n}(x,t_{n}) \dd  m^{n}(t_{k})$$			           
\end{proof}
\begin{remark}\label{extensionahamiltonianosmasgenerales}
{\rm(i)} As the proof shows, the costly assumption $h_{n}=o(\eps_{n}^{2})$ comes from the a priori non regularity of $Dv[\bar{m}](x,t)$ w.r.t. the time variable. In fact,   an argument similar to the one  used for the convergence in \eqref{nuevolimite}  cannot be applied since a priori $Dv[\bar{m}](x,\cdot)$ is not necessarily Riemman integrable and hence \eqref{unaexp} seems to be necessary. \smallskip\\
{\rm(ii)}  All the results of this paper, can be extended for the more general Hamiltonians $H(x,t,p)$ considered in {\rm \cite{AchdouCamilliCorrias11}}. In fact, consider the system 
\small
\be\ba{rcl}\label{MFGestoconH} 
-\partial_{t} v  -  \half \mbox{{\rm tr}}\left( \sigma(t) \sigma(t)^{\top} D^{2} v \right)  + H(x,t, Dv)   &=& F(x, m(t)) \;  \;    \hbox{{\rm in} } \RR^{d}\times ]0,T[, \\[6pt]
\partial_{t} m -\half  \mbox{{\rm tr}}\left( \sigma(t) \sigma(t)^{\top} D^{2} v \right)   -\mbox{{\rm div}} \big( \partial_{p} H(x,t, Dv)  m  \big)  &=&0 \;  \; \; \hbox{{\rm in} } \RR^{d}\times ]0,T[, \\[6pt]
v(x,T)= G(x, m(T)) \; \;   \mbox{{\rm for} } x \in \RR^{d}, &\,& \; \; m(\cdot,0)=m_0(\cdot) \in \P_{1}(\RR^{d}).
\ea\ee \normalsize
If the assumptions in  {\rm\cite[Section 2]{AchdouCamilliCorrias11}} for the Hamiltonian $H(x,t,p)$ hold true and for every $\mu \in C([0,T]; \P_{1}(\RR^{d}))$ the  $(OSL_h^{\rho})$ condition in  {\rm\cite[page 16]{AchdouCamilliCorrias11}}  is verified for $-\partial_{p} H(x,t,D v^{\eps}_{\rho,h}[\mu])$ {\rm(}where $v^{\eps}_{\rho,h}[\mu]$ is the Semi-Lagrangian approximation of the viscosity solution $v[\mu]$ of the HJB equation in \eqref{MFGestoconH}, with $m$ replaced by $\mu${\rm)},  then the proofs of this article can be reproduced for this more general case. 
\end{remark}

\section{Numerical Tests}\label{numericaltests}
We present some numerical simulations for the one dimensional case. For an easier explanation of the tests, let us recall the  heuristic interpretation of the MFG system:  an  average player, whose dynamic is given by
$$
dX(s)=\alpha(s)ds +\sigma(s)dW(s),\;{\mbox {for all}}\;  t\in[0,T], \hspace{0.2cm} X(0)=x\in \RR,
$$
and $W(\cdot)$  a standard one dimensional Brownian motion,
aims to minimize, with respect to the control $\alpha(\cdot)$,  the functional :
$$\EE \left( \int_0^ T\left[ \frac{1}{2}\alpha^2(s)+F(X(s),m(s))\right] \dd s +G(X^{x,t}(T),m(T))\right). $$
%The term $$\frac{1}{2}\alpha^2(t)+F(x,m(t))$$ is the running cost.
We will consider running costs of the form 
$$\frac{1}{2}\alpha^2+F(x,m)=\frac{1}{2}\alpha^2+f(x)+ V_{\delta}(x,m),$$
where $f$ is $C^2$ and 
\begin{equation}\label{eq:V}
V_{\delta}(x,m)= \phi_{\delta} \ast \left[\phi_{\delta}  \ast m \right](x) \hspace{0.3cm} \mbox{and } \; \; \; \phi_{\delta}(x)=\frac{1}{\sqrt{2 \pi}} e^{-x^2/(2 \delta^2)}. \end{equation}
for some $\delta>0$ to be chosen later.
%\be\label{eq:h_eps}
%\phi_{\delta}(x)=\frac{1}{\sqrt{2 \pi}} e^{-x^2/(2 \delta^2)}.
%\ee
We solve heuristically the fully discrete MFG system  \eqref{eqoqeqqrrsdsrrrooaaaaaaaaaaa} by a fixed-point iteration method. At a generic  iteration $p$,
let  us call $$\{(v^{\eps,p}_{i,k}, m^{\eps,p}_{i,k}),\; i\in\ZZ, k=0,\dots N\}_{p\in \NN}$$  the sequences representing the approximated value function and  mass distribution. We consider as initial guess %Omitting $\mu$ to simplify notations,  given 
  $$m^{\eps,0}_{i,k}=m^{\eps}_{i,0}=\int_{E_i}m_0(x)\dd x,\quad i\in \mathbb{Z},\; k=0,\dots,N.$$
Given $m^{\eps,p}_{i,k}$ we calculate $m^{\eps,p+1}_{i,k}$ according to the following scheme
 $$m^{\eps,p}_{i,k}\longrightarrow v^{\eps,p}_{i,k}\longrightarrow Dv^{\eps,p}(x_i,t_k)\longrightarrow m^{\eps,p+1}_{i,k},$$
where in the step
 $m^{\eps,p}_{i,k}\longrightarrow v^{\eps,p}_{i,k}$ we compute $\{v^{\eps,p}_{i,k}\}_{i,k}$ by solving the scheme  \eqref{dwoewoeoweaap} with discrete mass distribution given by $\{m^{\eps,p}_{i,k}\}_{i,k} $. In  the step $v^{\eps,p}_{i,k}\longrightarrow Dv^{\eps,p}(x_i,t_k)$ we compute  the discrete gradient of  $v^{\eps,p}$ by  approximating \eqref{ammrmooroorotttpsps} using a discrete convolution 
and then approximating  the gradient by central finite differences. In the last step
 $Dv^{\eps}(x_i,t_k)\longrightarrow m^{\eps,p+1}_{i,k}$, we compute $m^{\eps,p+1}_{i,k}$  by the  scheme \eqref{eqoooaaaaaaaaaaa}. We stop the fixed point method when the  errors 
 \be\label{errors}E(v^{\eps,p}):= \|v^{\eps,p+1}-v^{\eps,p}\|_{\infty},\quad E(m^{\eps,p}):= \|m^{\eps,p+1} -m^{\eps,p}\|_{\infty},\ee
are below a given threshold $\tau$ or $p$  has reached a fixed number of iterations. \smallskip

 So far, we have set the problem in the space domain $Q=\RR$. Clearly to implement the numerical scheme we have to suppose that the domain $Q$ is bounded. Following  \cite[Section 3]{CamFal95}, we will thus formally  constraint the problem to a sufficiently large bounded domain $Q_b$ by supposing now that $\sigma=\xi_{b}^{2}(x) \sigma(t)$, where $\xi_{b}\in C^{\infty}_{0}(\RR)$ satisfies $\xi_{b}(x)=1$ if $x\in Q_{b}$.  Note that by doing this we are imposing a dependence on $x$ for $\sigma$ and our results do not apply.  Moreover, for the Fokker Planck equation, in order to maintain the mass $m$ in $Q_b$, we will impose Neumann boundary conditions, which are not covered by our results neither.  Therefore,  the numerical resolution of the scheme is heuristic. However,  since we will consider cost functions that incite the players to remain on a bounded domain, this type of approximation is reasonable since   the  influence in the cost, expressed through $V_{\delta}(x,m)$, of players being far from $Q_b$, is negligible.   
 
% $Q_b$. 
 %Even if in general the characteristics,  solution of $dX=\alpha ds +\sigma dW $ may have an infinite exit time from $Q_b$, 
%Let us recall that for any fixed discretization parameters $(\rho,h,\eps)$,  the discrete sequence $m^{\eps}_{i,k}$ has a compact support $K_h$, see Remark \ref{mwrmmwrmmtmtmtmtttt}. Then 
%we can set our numerical computation on a bounded domain $Q_b$, large enough to contain almost the totality of the mass
%density during the whole evolution. We have then supposed to know  the compact set $Q_b$ and we have imposed zero Neumann condition at the boundary in the case of the Fokker Planck equation. For the Hamilton Jacobi equation, we can reduce the problem to an arbitrary bounded set, . \\
%the exit time for the discrete characteristics \eqref{flujofullydiscretounpaso} is finite. 
%then 
%We treat boundary conditions as in \cite{FF98}, where only on the nodes belonging to a neighborhood of the boundary a variable step is computed, in order to  keep  the  discrete characteristics  inside the bounded domain $Q_b$. \\
We will show three numerical tests, comparing the different behavior  at different choices for the diffusion term. First  we  consider the case in which the diffusion term is zero (studied already in \cite{CS12}), which corresponds to a deterministic MFG system, then the case with a constant and positive diffusion term, which corresponds to second order MFG system (see \cite{CS13}). Finally,  we consider the case where the  diffusion term is given by  a positive continuous function, which degenerates  in a given time interval.

%%%%%%%%%%%%%%
{ \bf{Test 1 (deterministic case)}} We consider a numerical domain $Q_b\times [0,T]=[0,1]\times [0,2]$ and 
  we choose  as initial mass distribution: 
$$m_0(x)=\frac{\nu(x)}{\int_{\Omega}\nu(x)dx} \; \, \;  {\rm with} \; \;  \nu(x)= e^{-(x-0.5)^2/ (0.1)^2}.$$
We choose as final cost $G=0$,  as  running cost $\frac{1}{2}\alpha^2(t)+f(x)+ V_{\delta}(x,m(t))$ with $\delta=0.2$ and
$$f(x,t)=5(x-(1-\sin(2\pi t))/2)^2,$$
and  we set  $\sigma(\cdot)\equiv 0$. 
%note that $\sigma(t)=0$  for all $t\in [0,0.8]\cup[1.2,2]$.
 % and is different than zero only in the time interval $[0.8,1.2]$.
 In the running cost  the term  $f(x,t)$ incites  the agents to stay close to the point $(1-\sin(2\pi t))/2$ at each time $t$, while the term $V_{\delta}(x,m)$ penalizes  high concentration of the density distribution.
The density evolution is shown in Fig.\ref{Test3mass}, which has been computed with $\rho=3.12\cdot 10^{-3} , h=\rho,\eps=0.15$. The number of iterations required by the fixed point method to satisfy the stopping criteria with $\tau=10^{-3} $ is 10.
 We observe, during the whole time interval, that the mass density tends to concentrate around  to the curve $(1-\sin(2\pi t))/2$ and no diffusion effect appears. It is important to remark that the term $V_{\delta}(x,m)$ has a non negligible effect in the distribution. As a matter of fact, if this term is not present, then much higher concentrations  are observed (see e.g. \cite[Fig. 4.8]{CS12}).
 \begin{figure}[ht!]
\begin{center}
\includegraphics[width=6cm]{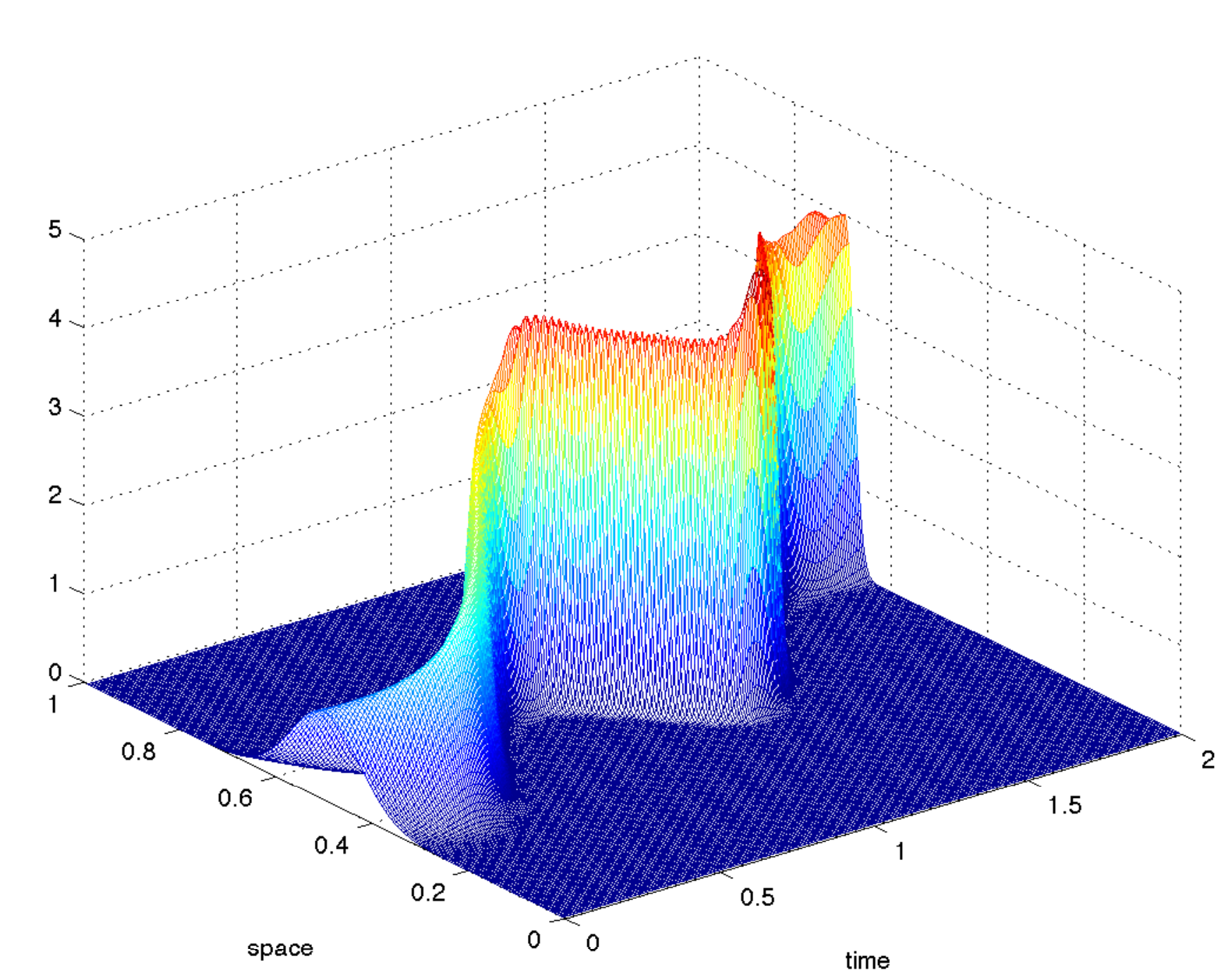}\includegraphics[width=6cm]{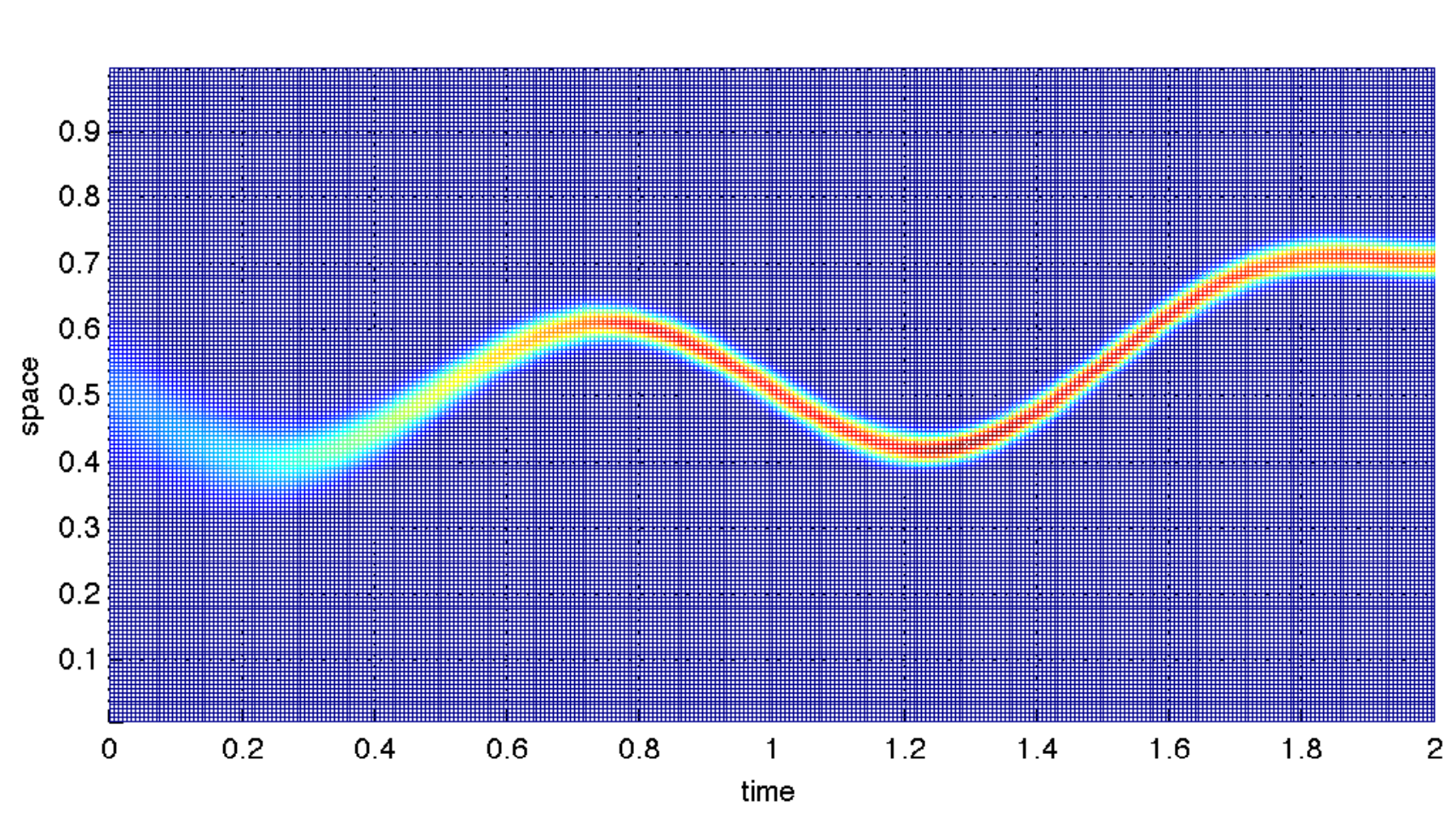}
\caption{{\bf Test 1: Mass evolution $m_{i,k}^\eps$}}
\label{Test3mass}
\end{center}
\end{figure}\\

{ \bf{Test 2 (non-degenerate diffusion)}} 
We consider the same problem as in Test 1, but now we change the diffusion term choosing $\sigma = 0.2$. Let us note that, in this case, the scheme reduce to the one proposed in \cite{CS13}.
The running cost and the initial distribution are chosen as in the previous tests.
%Note that $\sigma(t)=0$  for all $t\in [0,0.8]\cup[1.2,2]$.
 % and is different than zero only in the time interval $[0.8,1.2]$.
The density evolution is shown in Fig.  \ref{Test2mass}, which has been computed with $\rho=6.35\cdot 10^{-3} , h=\rho,\eps=2 \sqrt{h} $ and $\tau=10^{-3} $. The number of iterations for the fixed point method, to satisfy the stopping criteria with $\tau=10^{-3} $, is 6. Let us note that in this case the convergence  is faster compared to the deterministic case in Test 1.
A diffusive effect is observed during the whole time interval, which seems not very strong, since it is opposite to the one due to the running cost, which tends to concentrate the mass density around the sinusoidal curve.\\
 \begin{figure}[ht!]
\begin{center}
\includegraphics[width=6cm]{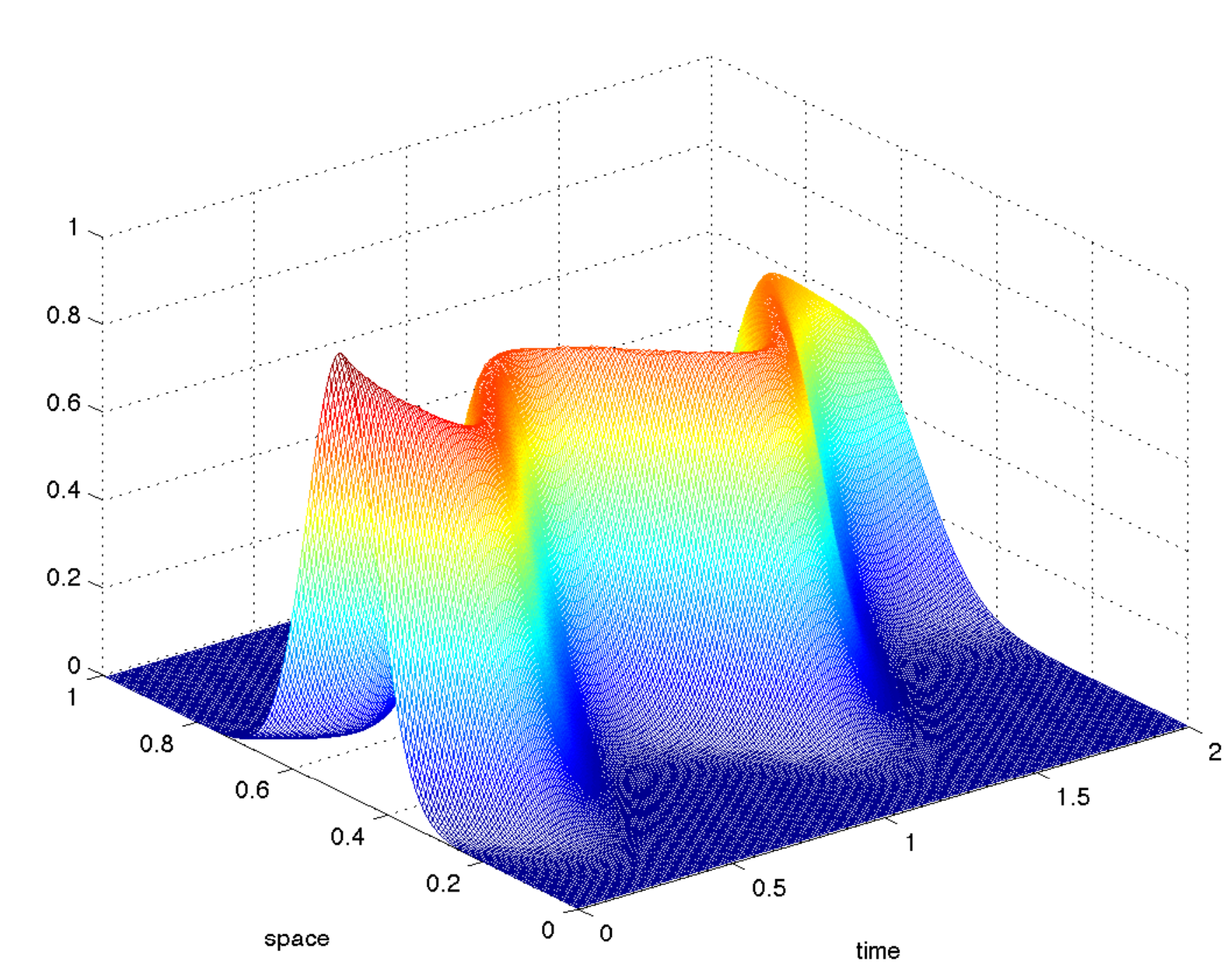}\includegraphics[width=6cm]{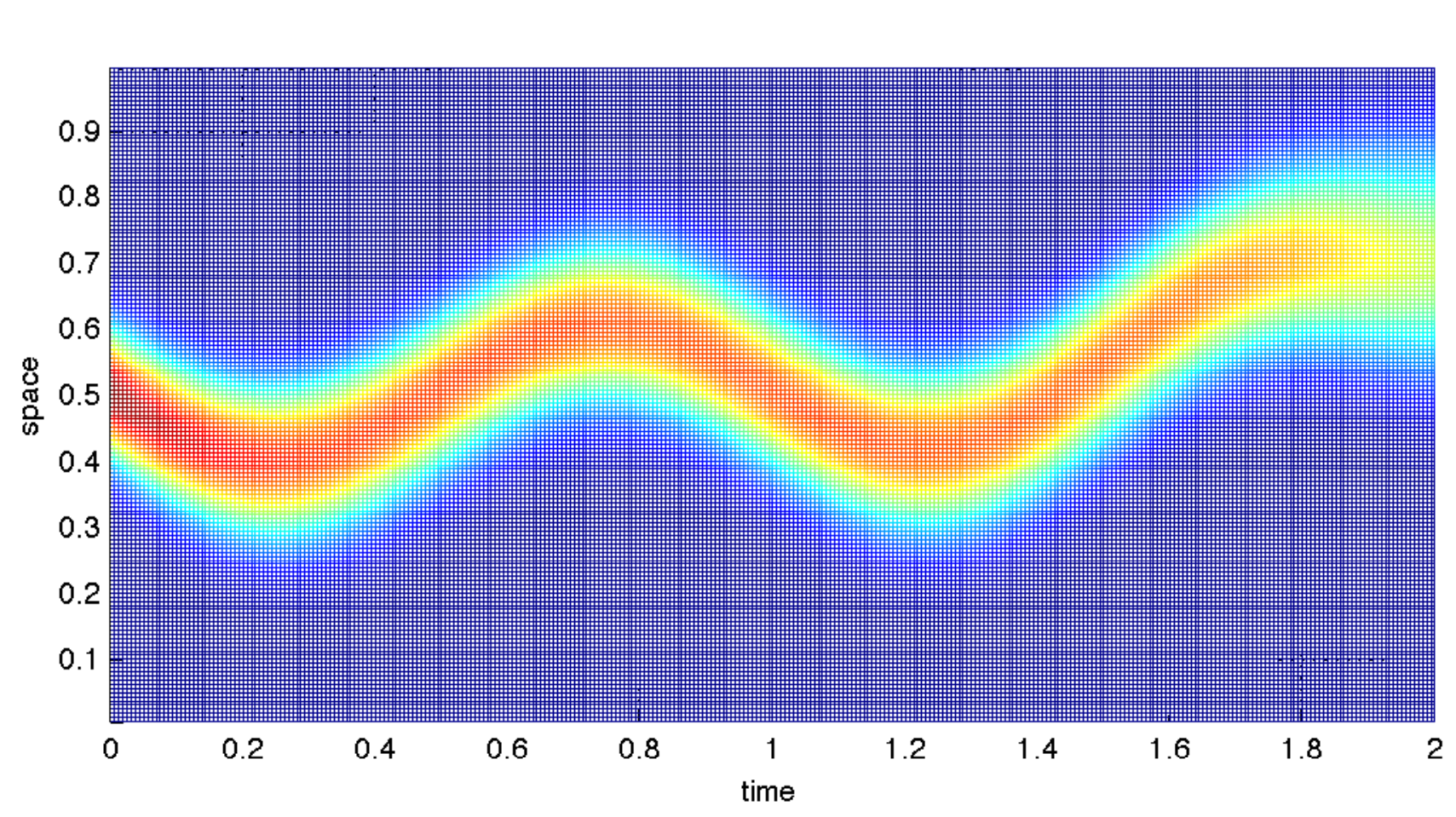}
\caption{{\bf Test 2: Mass evolution $m_{i,k}^\eps$ }}
\label{Test2mass}
\end{center}
\end{figure}\\

{ \bf{Test 3 (degenerate diffusion)}} 
We consider the same problem as in Test 1, but now we change the diffusion term choosing 
a  scalar function 
$$\sigma(t)=\max(0,0.2-|t-1|).$$
Note that $\sigma(t)=0$  for all $t\in [0,0.8]\cup[1.2,2]$.
The running cost and the initial distribution are chosen as in the previous tests.
%Note that $\sigma(t)=0$  for all $t\in [0,0.8]\cup[1.2,2]$.
 % and is different than zero only in the time interval $[0.8,1.2]$.
 The density evolution is shown in Fig.  \ref{Test1mass}, which has been computed with $\rho=6.35\cdot 10^{-3} , h=\rho,\eps=2 \sqrt{h} $ and $\tau=10^{-3} $.  The number of iterations, for the fixed point method to satisfy the stopping criteria with $\tau=10^{-3} $, is 9. Let us note that in this case the rate of convergence, for the fixed point method, is  between the rates for the  two cases.
 We observe a diffusive effect  during the time interval $[0.8,1.2]$, due to the non zero term  $\sigma(t)$. When the diffusion stops to act, a time $t=1.2$ the density starts again to concentrate faster around the curve where $f$ is lower. 
\begin{figure}[ht!]
\begin{center}
\includegraphics[width=6cm]{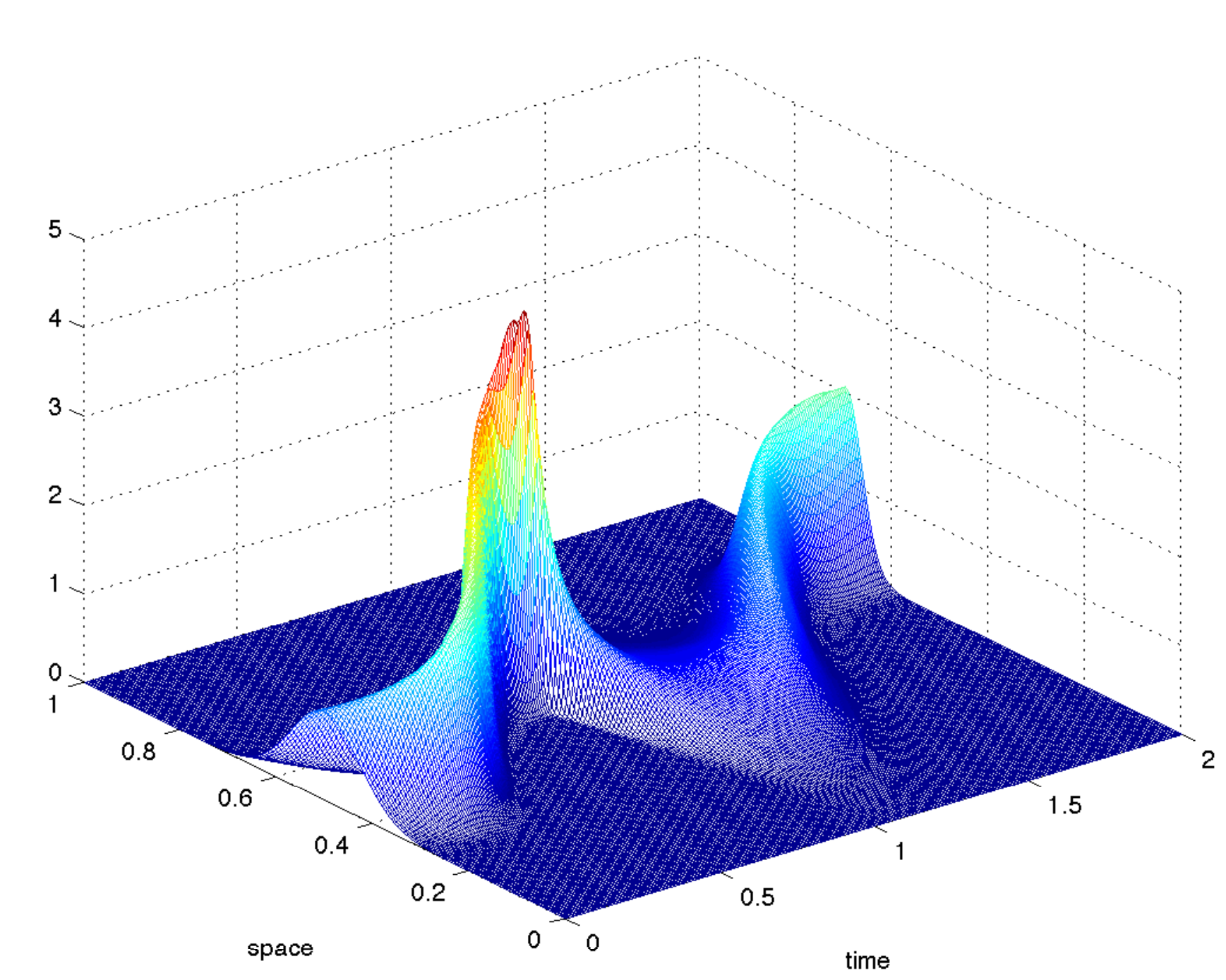}\includegraphics[width=6cm]{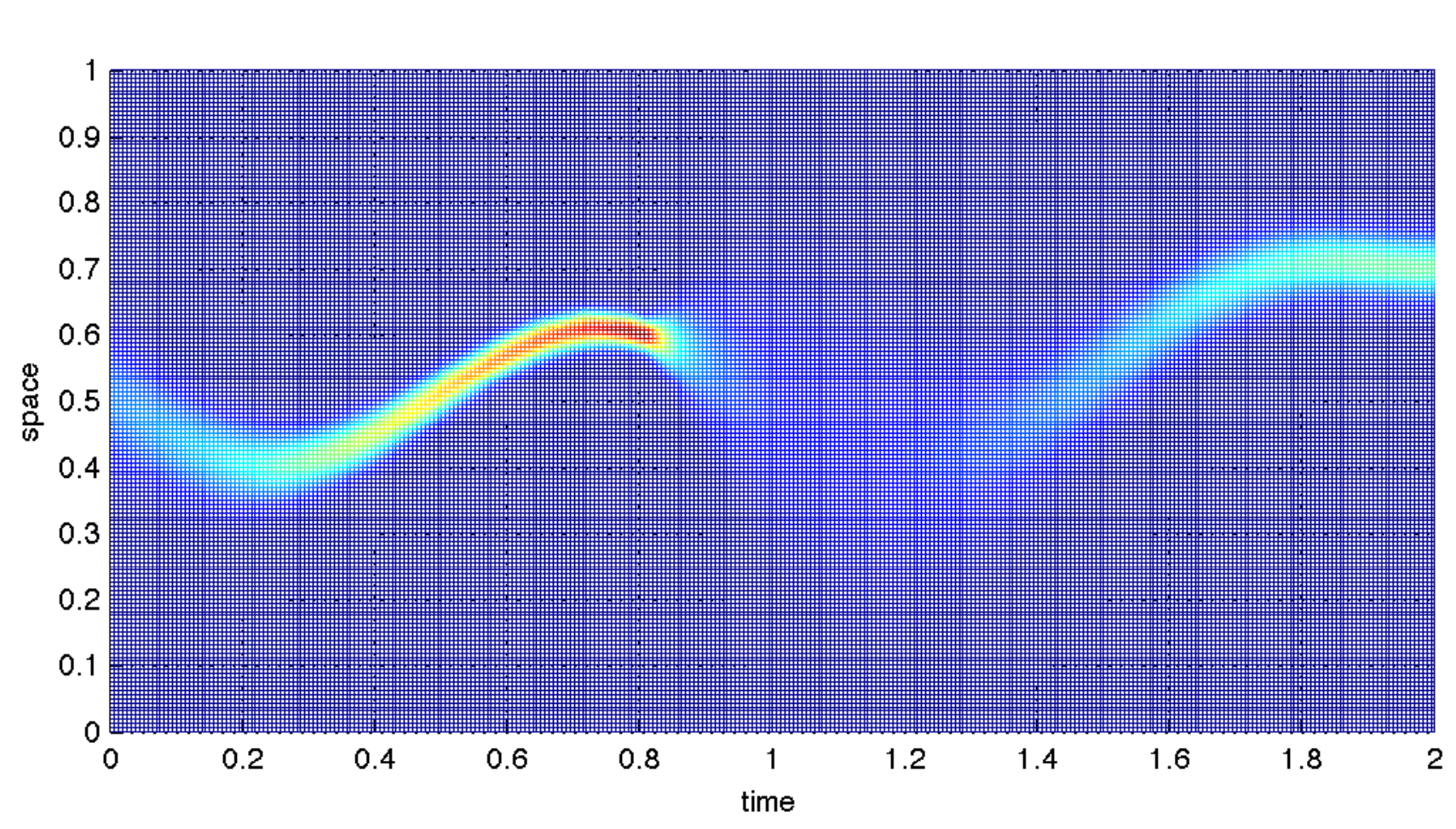}
\caption{{\bf Test 1: Mass evolution $m_{i,k}^\eps$}}
\label{Test1mass}
\end{center}
\end{figure}\\
%\begin{figure}[ht!]
%\begin{center}
%\epsfig{figure=figure/test1uhj.eps,width=5cm}\epsfig{figure=figure/test1controllo.eps,width=5cm}\\
%\caption{ {\bf{ Value function }} $v^{\eps}_{i,k}$ (left) and {\bf Gradient } $D v^{\eps}_{i,k}$(right). }
%\label{Test1valuefandoptf}
%\end{center}
%\end{figure}
Table \ref{tab:test1SL} shows  the errors \eqref{errors} computed  varying  all the parameters $(\rho,h,\eps)$, according the balance $h=\rho$ and $\eps=2 \sqrt{h}  $.  In the first two columns of Table \ref{tab:test1SL} we show the space and regularizing parameters, in the last two columns the errors for the value function and the density computed after 10 iterations of the fixed point algorithm.\\
\begin{table}[ht!!]\caption{Parameters and errors }
\begin{center}\label{tab:test1SL}
\begin{tabular}{|c|c|c|c|c|}\hline
$\rho$ & $\varepsilon $ &$E(v^{\eps,10})$& $E(m^{\eps,10})$ \\
\hline \hline
$1.25\cdot 10^{-2}$& $0.2$& $1.72 \cdot 10^{-6} $& $ 9.52\cdot 10^{-5} $\\ \hline 
$6.25\cdot 10^{-3}$& $0.15$ & $ 1.08\cdot 10^{-6} $&  $ 1.17 \cdot 10^{-4} $ \\ \hline 
$3.12\cdot 10^{-3}$& $0.1 $& $1.82\cdot 10^{-6} $ & $3.26\cdot 10^{-4} $\\ \hline  
%$1.87\cdot 10^{-3}$&$3.75 \cdot 10^{-3}$& $1.60 \cdot 10^{-2} $& $9.74 \cdot 10^{-4} $ &$3.56\cdot 10^{-3} $\\ \hline 
\end{tabular}\\[30pt]
\end{center}
\end{table}

In Fig. \ref{Test1parvar}, we show the behavior of the errors \eqref{errors} in logarithmic scale on the $y$-axes versus the number of fixed-point iterations on the $x$-axes. We vary all the parameters according to the Table \ref{tab:test1SL}.

\begin{figure}[ht!]
\begin{center}
\includegraphics[width=5cm]{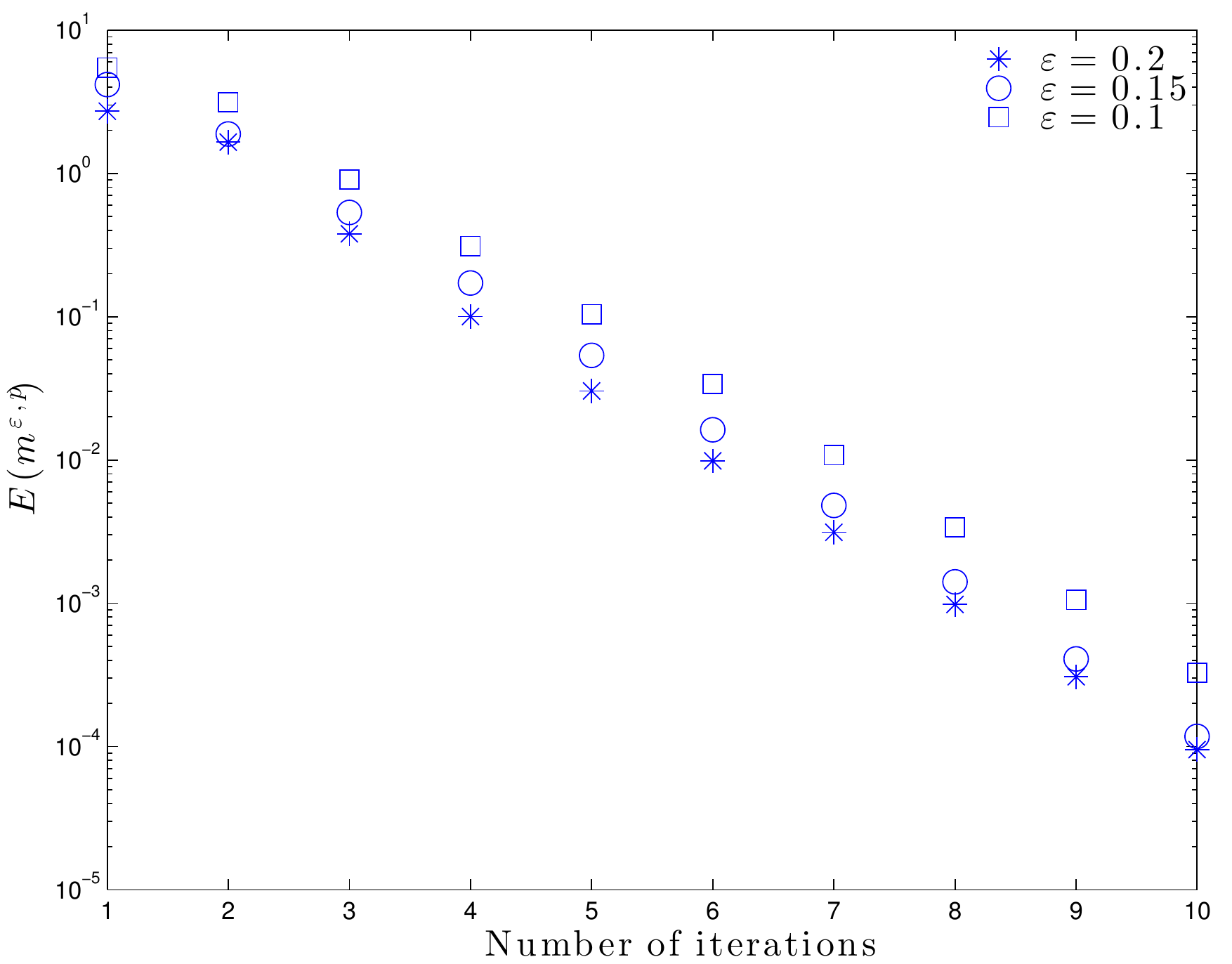}\includegraphics[width=5cm]{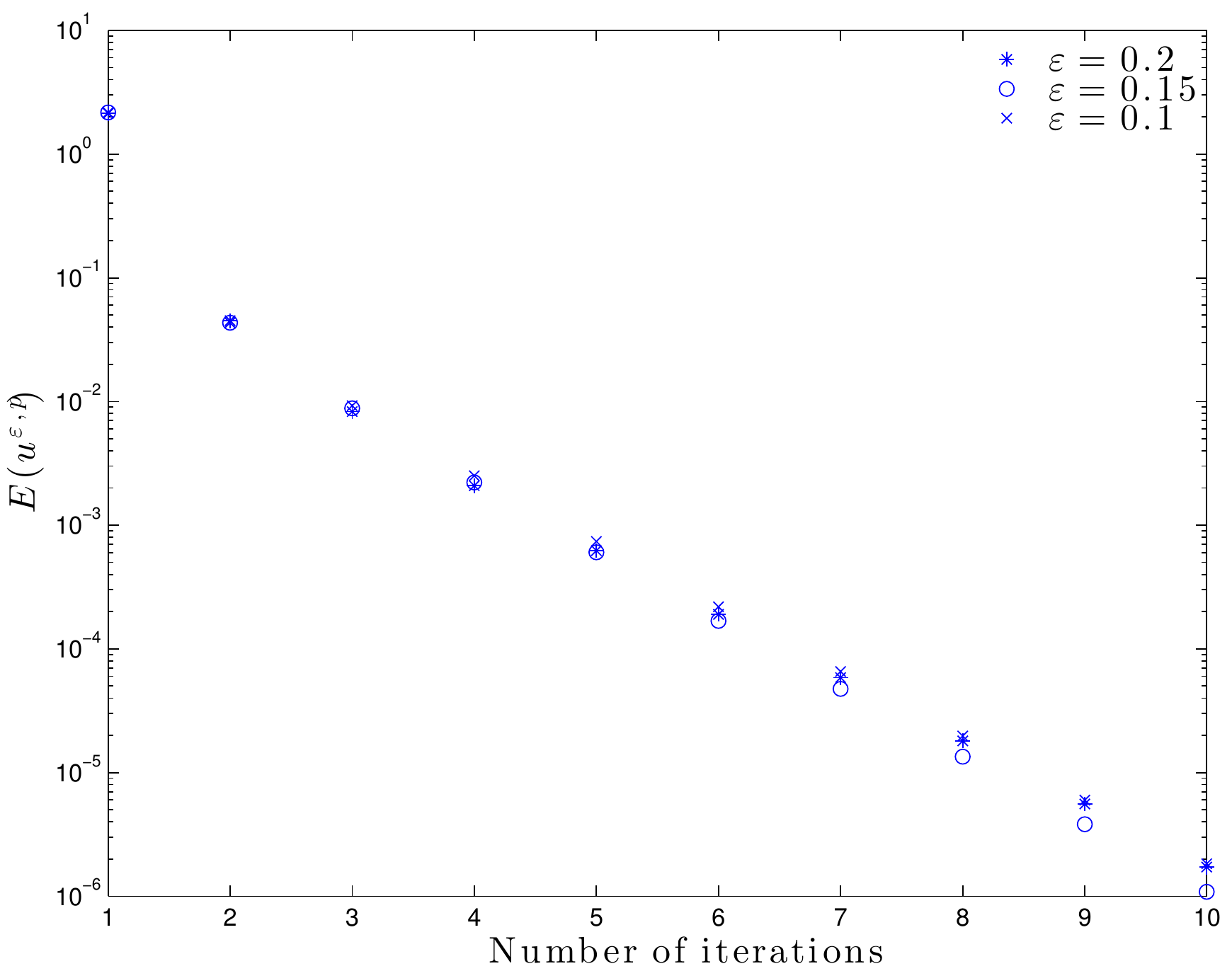}
\caption{ {\bf Errors}: $E(m^{\eps,p})$ (left) $E(u^{\eps,p})$ (right) varying all the parameters $(\eps,\rho,h)$ according to  Table \ref{tab:test1SL}.\label{Test1parvar}}
\end{center}
\end{figure}
%%%%%%%%%%%%%%

 \bibliographystyle{plain}
%\bibliography{bibpostdocultimo}
\bibliography{bibMFG}
\end{document}